\documentclass{amsart} 
\usepackage{amsthm, amsmath, mathtools}
\usepackage{amssymb}
\usepackage{amscd}
\usepackage{enumerate}
\usepackage{mathrsfs} 
\usepackage[curve,matrix,arrow,frame,tips]{xy}
\usepackage{colonequals}
\usepackage{verbatim}
\usepackage{pdfsync}
\usepackage{graphicx}
\usepackage[usenames,dvipsnames]{color}

\usepackage[curve,matrix,arrow,frame,tips]{xy}

\RequirePackage{tikz-cd}
\usepackage{tikz}

\usepackage{url}

\usepackage{fullpage}

\usepackage[normalem]{ulem}

\usepackage{hyperref}

\numberwithin{equation}{section} 

\theoremstyle{plain}
\newtheorem{theorem}{Theorem}[section]
\newtheorem{proposition}[theorem]{Proposition}
\newtheorem{lemma}[theorem]{Lemma}

\newtheorem{corollary}[theorem]{Corollary}

\theoremstyle{definition}
\newtheorem{definition}[theorem]{Definition}

\newtheorem{nota}[theorem]{Notation}

\theoremstyle{remark}
\newtheorem{remark}[theorem]{Remark}

\newcommand{\C}{\mathbb{C}}

\newcommand{\Q}{\mathbb{Q}}

\newcommand{\Z}{\mathbb{Z}}

\DeclareMathOperator{\Spec}{Spec}

\DeclareMathOperator{\SL}{SL}
\DeclareMathOperator{\GL}{GL}

\DeclareMathOperator{\Gal}{Gal}

\renewcommand{\to}{\longrightarrow}

\newcommand{\HH}{\mathrm{H}}

\usepackage{microtype}

\begin{document}

\title{Triple Massey products for higher genus curves}

\author[F. Bleher]{Frauke M. Bleher}
\address{F.B.: Department of Mathematics\\University of Iowa\\
14 MacLean Hall\\Iowa City, IA 52242-1419\\ U.S.A.}
\email{frauke-bleher@uiowa.edu}
\thanks{The first author was supported in part by Simons Foundation grant No. 960170.}

\author[T. Chinburg]{Ted Chinburg}
\address{T.C.: Department of Mathematics\\University of Pennsylvania\\
Philadelphia, PA 19104-6395\\ U.S.A.}
\email{ted@math.upenn.edu}
\thanks{The second author was supported in part by Simons Foundation grant No. MP-TSM-00002279.}

\author[J. Gillibert]{Jean Gillibert}
\address{J. G.:  Institut de Math{\'e}matiques de Toulouse \\ CNRS UMR 5219 \\
118, route de Narbonne, 31062 Toulouse Cedex \\ France.}
\email{Jean.Gillibert@math.univ-toulouse.fr}
\thanks{The third author was supported in part by the CIMI Labex.}

\date{\today}


\begin{abstract}
We study the vanishing of triple Massey products for absolutely irreducible smooth projective curves  over a number field. For each genus $g > 1$  and each prime $\ell > 3$, we construct examples of hyperelliptic curves of genus $g$ for which there are non-empty triple Massey products with  coefficients in $\mathbb{Z}/\ell$  that do not contain $0$.
\end{abstract}

\maketitle


\section{Introduction}
\label{s:intro}

Triple Massey products arise in topology from studying the higher linking behavior of closed curves in the three-sphere \cite{Massey1,Massey2}.  For example, the fact that the Borromean rings are linked is a consequence of the  non-vanishing of a triple Massey product.  The point of view we adopt in this paper is that the triple Massey product measures the complexity of the quotient of a group by the fourth term in its lower central series.  The precise definition of triple Massey products will be recalled in \S \ref{s:prelim}.  Suppose $\Gamma$ is a group and $\ell$ is a prime.  Let $U_4(\mathbb{Z}/\ell)$ be  the group of upper triangular unipotent  $4 \times 4$ matrices over $\mathbb{Z}/\ell$, and let $Z(U_4(\mathbb{Z}/\ell))$ be its center.  As explained in \S \ref{ss:tripleMasseygeneral}, there is a non-empty  triple Massey product for $\Gamma$ with coefficients in  $\mathbb{Z}/\ell $ that does not contain $0$ exactly when there exists a homomorphism 
$\tilde{\rho}:\Gamma \to U_4(\mathbb{Z}/\ell)/Z(U_4(\mathbb{Z}/\ell))$ with the following property.  Suppose $\tilde{\rho}' :\Gamma \to U_4(\mathbb{Z}/\ell)/Z(U_4(\mathbb{Z}/\ell))$ is any homomorphism such that $\tilde{\rho}$ and $\tilde{\rho}'$ produce the same homomorphism from $\Gamma$ to the quotient of $U_4(\mathbb{Z}/\ell)$ by its commutator subgroup.  Then $\tilde{\rho}'$ cannot be lifted to a homomorphism
$\rho':\Gamma \to U_4(\mathbb{Z}/\ell)$.   Our main result is as follows:

\begin{theorem}
\label{thm:TedsShamefulVersion}
Let $\ell>3$ be a prime number, and let $g>1$ be an integer. Then there exist infinitely many non-isomorphic pairs $(F,X)$ consisting of a number field $F$ and a smooth genus $g$ hyperelliptic curve $X$ over $F$ having the following two properties:
\begin{itemize}
\item[(i)] The $\ell$-torsion of the Jacobian of $X$ over an algebraic closure of $F$  is already defined over $F$.
\item[(ii)] There is a non-empty triple Massey product for the \'etale fundamental group $\Gamma$ of $X$ with coefficients in $\mathbb{Z}/\ell$ that does not contain $0$.
\end{itemize}
\end{theorem}

We in fact prove a stronger result in Theorem \ref{thm:JeanVersion}.  An easy Chebotarev argument shows that Theorem \ref{thm:TedsShamefulVersion} also holds with $F$ being a finite field in which $\ell$ is invertible rather than a number field (see Remark \ref{rem:Chebotarev}).

Suppose $X$ is a smooth projective curve over a field $F$ in which $\ell $ is invertible and that $X$ has positive genus.  Let $\pi_1(X,\eta)$ be the \'etale fundamental group of $X$ with respect to a geometric base point $\eta$.
By a result of Achinger (see Proposition \ref{prop:Achinger}), there is a natural isomorphism for all $i \ge 0$ between the \'etale cohomology group $\HH^i(X,\mathbb{Z}/\ell)$ and the profinite group cohomology $\HH^i(\pi_1(X,\eta),\mathbb{Z}/\ell)$.
Thus one could rephrase our results in terms of triple Massey products associated to triples of elements of
$\HH^1(X,\mathbb{Z}/\ell)$.

We now describe some prior work on triple Massey products  for $\HH^1(X,\mathbb{Z}/\ell)$ when $\ell > 2$ and $X$ is an absolutely irreducible smooth projective variety over a  field $F$ in which $\ell$ is invertible.  When $d = \mathrm{dim}(X) = 0$, Min{\'a}\v{c} and T{\^a}n showed in \cite{MinacTan2016} that for arbitrary $F$, all non-empty triple Massey products contain $0$, following earlier work by
 Hopkins and Wickelgren \cite{HopkinsWickelgren}, Matzri \cite{Matzri}, Efrat and Matzri \cite{EfratMatzri}, and others.  When $d = 0$ and $F$ is a number field, Harpaz and Wittenberg showed in \cite{HarWit2019} that all non-empty $t$-fold Massey products contain $0$ when  $t \ge 3$.  For a more detailed account of the case $d  = 0$ see the introduction of \cite{HarWit2019}. 
 
 Ekedahl gave an example in \cite{Ekedahl} showing that the triple Massey product of elements of 
$\HH^1(X,\mathbb{Z}/\ell)$ may be non-empty without containing $0$ when $d = 2$ and $F = \mathbb{C}$.

 In \cite{BleherChinburgGillibert2023} we considered the case in which $d = 1$ and $X$ is an elliptic curve over a field $F$ in which $\ell$ is invertible.  When the $\ell$-torsion of $X \cong \mathrm{Jac}(X)$ over an algebraic closure of $F$ is already defined over $F$ and $\ell > 3$, we proved that non-empty triple Massey products always contain $0$.  When $\ell = 3$ we showed they need not contain $0$ and we analyzed exactly when this occurs. One interest of Theorem \ref{thm:TedsShamefulVersion} is that for curves of genus larger than $1$, triple Massey products can be non-empty and not contain $0$ for arbitrary primes $\ell > 3$ even when the $\ell$-torsion of the Jacobian of the curve is defined over the ground field. 
 
 We now outline the contents of this paper.  
 
 In \S \ref{s:prelim} we recall some basic definitions and set up  notations.  The main result of \S \ref{s:prelim} is Proposition \ref{prop:sufficient_general}.  This gives group theoretic sufficient conditions for the existence of a non-empty triple Massey product that does not contain $0$ associated to a group $\Gamma$ with coefficients in $\mathbb{Z}/\ell$ when $\ell > 3$.  
 Our technique for proving Theorem \ref{thm:TedsShamefulVersion} and the stronger result in Theorem \ref{thm:JeanVersion} is to produce examples for which the above sufficient conditions are met.  To do this, we recall in \S \ref{s:topology} some results of Mumford,  Farb and Morita concerning normalized pure braid groups and mapping class groups.  The normalized pure braid group that arises is the fundamental group of  an affine rational parameter space $\mathcal{H}_g^{(2)}$ for a family 
$\mathcal{C}$ of hyperelliptic curves of genus $g$.   Our goal is to show there exists a fiber $X$ in the family $\mathcal{C}$, defined over a number field $F$,  for which the criteria of Proposition \ref{prop:sufficient_general} are met, thus giving a non-empty triple Massey product that does not contain $0$.

 To accomplish this, we derive in \S \ref{ss:johnson} some consequences of work of  Morita concerning Johnson's homomorphisms.  These homomorphisms describe the actions of particular subgroups of the mapping class group of a Riemann surface on quotients of the fundamental group of the surface by terms in  its lower central series. We use work of Mumford, Farb and Morita to exhibit particular elements of the normalized pure braid group $\pi_1^{\mathrm{top}}(\mathcal{H}_g^{(2)})$ with prescribed action on a certain finite characteristic quotient of the fundamental group of a fiber of the family $\mathcal{C} \to \mathcal{H}_g^{(2)}$ of hyperelliptic curves.  These elements lie in particular cosets of  a finite index normal subgroup of $\pi_1^{\mathrm{top}}(\mathcal{H}_g^{(2)})$. 
  
 The finite quotient of $\pi_1^{\mathrm{top}}(\mathcal{H}_g^{(2)})$ by this normal subgroup leads to a finite Galois topological covering $B(4,\ell) \to \mathcal{H}_g^{(2)}$.  In \S \ref{ss:alg}, wee use Grothendieck's existence theorem to realize this covering as an irreducible cover of algebraic varieties over a number field $K$.  Since $\mathcal{H}_g^{(2)}$ is a rational variety, we can then use Hilbert's irreducibility theorem in \S \ref{ss:spe} to produce infinitely many $K$-rational points $t$ of $\mathcal{H}_g^{(2)}$  over which  this cover remains irreducible.  We specialize to these points $t$ the above results about the action of particular elements in $\pi_1^{\mathrm{top}}(\mathcal{H}_g^{(2)})$ on the fundamental groups of the fiber of $\mathcal{C} \to \mathcal{H}_g^{(2)}$ over $t$.  
This leads to  examples for which the criteria of Proposition \ref{prop:sufficient_general}  are met, leading to a proof of Theorem \ref{thm:TedsShamefulVersion}. 


\section{Preliminaries}
\label{s:prelim}

In this section, we recall some definitions and results that will be needed for the remainder of the paper. Throughout the paper, we assume that $\ell > 3$ is a fixed rational prime.

\subsection{Lower central series and unitriangular groups}
\label{ss:lowercentral}

We first introduce some notation and definitions.

\begin{nota}
\label{not:groupgeneration}
Let $\Omega$ be an abstract (resp. profinite) group. Let $x,y\in \Omega$, and let $A,B\subseteq \Omega$. 
\begin{enumerate}
\item[(a)] We define the commutator of $x$ and $y$ to be $[x,y]=xyx^{-1}y^{-1}$.
\item[(b)] We denote by $\langle A\rangle$ the subgroup of $\Omega$ that is generated (resp. profinitely generated) by $A$, i.e. $\langle A\rangle$ is the smallest subgroup (resp. the smallest \emph{closed} subgroup) of $\Omega$ containing $A$.
\item[(c)] Suppose $V$ is a set (resp. a topological space), and $\tilde{\Omega}$ is an abstract (resp. profinite) group. Suppose $R$ is a commutative ring on which $\Omega$ acts trivially, where we give $R$ the discrete topology when $\Omega$ is a profinite group.
\begin{itemize}
\item[(c1)] We denote by $\mathrm{Map}(\Omega,V)$ the set of all maps (resp. continuous maps) from $\Omega$ to $V$.
\item[(c2)] We denote by $\mathrm{Hom}(\Omega,\tilde{\Omega})$ the set of all group homomorphisms (resp. continuous group homorphisms) from $\Omega$ to $\tilde{\Omega}$. Similarly, we denote by $\mathrm{Aut}(\Omega)$ the group of all group automorphisms (resp. continuous group automorphisms) of $\Omega$.
\item[(c3)] We denote by $\HH^i(\Omega,R)$ (for $i\ge 0$) the $i$-th group cohomology (resp. the $i$-th continuous group cohomology) of $\Omega$ with coefficients in $R$ (for $i\ge 0$). Since we assume $\Omega$ acts trivially on $R$, we have
$$\HH^1(\Omega,R)=\mathrm{Hom}(\Omega,R).$$
\end{itemize}
\item[(d)]
We also use the following notation:
\begin{eqnarray*}
[A,B] &=& \langle [a,b]\,:\, a\in A\mbox{ and } b\in B\rangle,\mbox{ and }\\
\Omega^\ell &=& \langle g^\ell \,:\, g\in \Omega\rangle.
\end{eqnarray*}
\end{enumerate}
\end{nota}

\begin{definition}
\label{def:lowercentral}
Let $\Omega$ and $G$ be both abstract (resp. profinite) groups, and assume Notation \ref{not:groupgeneration}.
\begin{enumerate}
\item[(a)] The \emph{lower central series} of $\Omega$ is given by 
$$\Omega=\Omega_1\ge \Omega_2\ge \Omega_3\ge \ldots$$
where $\Omega_{i+1} = [\Omega_i,\Omega]$ for $i\ge 1$.

\item[(b)] The \emph{lower central $\ell$-series} of $\Omega$ is given by 
$$\Omega=\Omega_{1,\ell}\ge \Omega_{2,\ell}\ge \Omega_{3,\ell}\ge \ldots$$
where $\Omega_{i+1,\ell}=\langle [\Omega_{i,\ell},\Omega]\cup (\Omega_{i,\ell})^\ell \rangle$ for $i\ge 1$.

\item[(c)] Suppose we have an action of $G$ on $\Omega$ given by a morphism
$$\varphi\in\mathrm{Hom}(G, \mathrm{Aut}(\Omega))$$
We write the elements of the semidirect product $\Omega \rtimes_\varphi G$ as pairs $(\omega,g)$ with $\omega\in \Omega$ and $g\in G$, and we identify $\Omega$ with the (closed) normal subgroup $\Omega\times 1$ and $G$ with the (closed) subgroup $1\times G$. Fix $i\ge 1$. Then $\varphi$ induces (continuous) group homomorphisms
$$\varphi_i:G\to\mathrm{Aut}(\Omega/\Omega_i) \quad\mbox{and}\quad \varphi_{i,\ell}:G\to\mathrm{Aut}(\Omega/\Omega_{i,\ell}).$$
We define
$$G(i) = \mathrm{kernel}\left(\varphi_i\right)\quad\mbox{and}\quad G(i,\ell) = \mathrm{kernel}\left(\varphi_{i,\ell}\right).$$
\end{enumerate}
\end{definition}

\begin{remark}
\label{rem:lowercentral}
Let $\Omega$ be an abstract (resp. profinite) group.
\begin{enumerate}
\item[(a)]
Since
$$[g_1g_2,h] = (g_1(g_2hg_2^{-1}h^{-1})g_1^{-1})(g_1hg_1^{-1}h^{-1}) = [g_1g_2g_1^{-1}, g_1hg_1^{-1}]\cdot [g_1,h]$$
for all $g_1,g_2,h\in \Omega$, it follows, using induction, that, for all $i\ge 2$,
$$\Omega_i=\langle [\cdots[[x_1,x_2],x_3],\cdots,x_i]\,:\, x_1,x_2,x_3\ldots,x_i\in \Omega\rangle.$$
In other words, $\Omega_i$ is generated by all $i$-fold commutators of elements of $\Omega$. In particular, this shows that for all $i\ge 1$,
$$(\Omega_2)_i\le \Omega_{i+1}.$$
Similarly, we have for all $i\ge 1$,
$$(\Omega_{2,\ell})_{i,\ell}\le \Omega_{i+1,\ell}.$$

\item[(b)] Suppose $\Omega$, $G$ and $\varphi$ are as in Definition \ref{def:lowercentral}(c) and that $\Omega$ is \emph{finitely generated} as abstract (resp. profinite) group. Using induction, it follows that $\Omega/\Omega_{i,\ell}$ is a finite group for all $i\ge 1$. 
In particular, $G(i,\ell)$ is a (closed) normal subgroup of $G$ of finite index for all $i\ge 1$.
\end{enumerate}
\end{remark}

\begin{remark}
\label{rem:Johnsonhomgeneral}
Let $\Omega$ and $G$ be both abstract (resp. profinite) groups, and assume Notation $\ref{not:groupgeneration}$ and Definition $\ref{def:lowercentral}$. Suppose $\varphi\in \mathrm{Hom}(G,\mathrm{Aut}(\Omega))$ is a group action and define for all $\Phi\in G$ and all $\omega\in\Omega$
$$\Phi.\omega:=\varphi(\Phi)(\omega).$$
Fix $i\ge 2$.
\begin{itemize}
\item[(a)] If $\Phi\in G(i,\ell)$ then the element $(\Phi.\omega)\,\omega^{-1}$ lies in $\Omega_{i,\ell}$ for all $\omega\in \Omega$. Moreover, for all $\omega_1,\omega_2\in\Omega$, we have
\begin{eqnarray*}
\left(\Phi.(\omega_1\omega_2)\right)\,(\omega_1\omega_2)^{-1}
&=&(\Phi.\omega_1)\,(\Phi.\omega_2)\,(\omega_2^{-1}\omega_1^{-1}) \\
&=& (\Phi.\omega_1)\,\omega_1^{-1} \, (\Phi.\omega_2)\,\omega_2^{-1}\, \left(\omega_2\,(\Phi.\omega_2)^{-1}\,\omega_1 (\Phi.\omega_2)\,\omega_2^{-1}\omega_1^{-1}\right)\\
&=&(\Phi.\omega_1)\,\omega_1^{-1} \, (\Phi.\omega_2)\,\omega_2^{-1}\, 
\left[\omega_2\,(\Phi.\omega_2)^{-1},\omega_1 \right] \\
&\equiv&((\Phi.\omega_1)\,\omega_1^{-1})\, ((\Phi.\omega_2)\,\omega_2^{-1})\mod \Omega_{i+1,\ell}
\end{eqnarray*}
where the last equivalence follows since $[\Omega_{i,\ell},\Omega]\le \Omega_{i+1,\ell}$.
Therefore, we obtain a well-defined (continuous) group homomorphism
$$\tau_{i,\ell}:G(i,\ell)\to \mathrm{Hom}(\Omega,\Omega_{i,\ell}/\Omega_{i+1,\ell})$$
by defining
$$\tau_{i,\ell}(\Phi)(\omega) = (\Phi.\omega)\,\omega^{-1}\mod \Omega_{i+1,\ell}$$
for all $\Phi\in G(i,\ell)$ and all $\omega\in \Omega$. 
Note that since $\Omega_{i,\ell}/\Omega_{i+1,\ell}$ is abelian of exponent $\ell$ and $\Omega/\Omega_{2,\ell}$ is the maximal abelian exponent $\ell$ quotient of $\Omega$, we can identify
$$\mathrm{Hom}(\Omega,\Omega_{i,\ell}/\Omega_{i+1,\ell})=
\mathrm{Hom}(\Omega/\Omega_{2,\ell},\Omega_{i,\ell}/\Omega_{i+1,\ell}).$$
Moreover, $\tau_{i,\ell}$ factors through the quotient group $G(i,\ell)/G(i+1,\ell)$.
\item[(b)]
Using $[\Omega_i,\Omega] = \Omega_{i+1}$, a similar argumentation to part (a) shows that we also have a well-defined (continuous) group homomorphism
$$\tau_i:G(i)\to\mathrm{Hom}(\Omega/\Omega_2,\Omega_i/\Omega_{i+1})$$
given by
$$\tau_i(\Phi)(\omega) = (\Phi.\omega)\,\omega^{-1}\mod \Omega_{i+1}$$
for all $\Phi\in G(i)$ and all $\omega\in \Omega$. Similarly to part (a), $\tau_i$ factors through the quotient group $G(i)/G(i+1)$.
\end{itemize}
\end{remark}

We next recall some well-known facts about unitriangular matrices with coefficients in $\mathbb{Z}/\ell$. 

\begin{remark}
\label{rem:U4}
Let $U_4(\mathbb{Z}/\ell)$ be the group of upper triangular unipotent (a.k.a. unitriangular) $4\times 4$ matrices with entries in $\mathbb{Z}/\ell$. Since we assume $\ell>3$, $U_4(\mathbb{Z}/\ell)$ has exponent $\ell$. Let
\begin{equation}
\label{eq:matrixnotation}
M=M(a_1,a_2,a_3,u,v,w):= \begin{pmatrix} 
1&a_1&u&v\\
0&1&a_2&w\\
0&0&1&a_3\\ 
0&0&0&1
\end{pmatrix}
\end{equation}
in $U_4(\mathbb{Z}/\ell)$ for $a_1,a_2,a_3,u,v,w\in\mathbb{Z}/\ell$, and let $\tilde{M}=M(\tilde{a}_1,\tilde{a}_2,\tilde{a}_3,\tilde{u},\tilde{v},\tilde{w})$. Then the commutator $[M,\tilde{M}]=M\tilde{M}M^{-1}\tilde{M}^{-1}$ equals
\begin{equation}
\label{eq:commutator}
[M,\tilde{M}] = 
\begin{pmatrix}
1&0&a_1\tilde{a}_2-a_2\tilde{a}_1&
(a_1\tilde{w}-w\tilde{a}_1)-(a_3\tilde{u}-\tilde{a}_3u) - (a_1\tilde{a}_2-a_2\tilde{a}_1)(a_3 + \tilde{a}_3)
\\
0&1&0&a_2\tilde{a}_3-a_3\tilde{a}_2\\
0&0&1&0\\ 
0&0&0&1
\end{pmatrix}.
\end{equation}

In particular, the commutator subgroup $[U_4(\mathbb{Z}/\ell),U_4(\mathbb{Z}/\ell)]$ of $U_4(\mathbb{Z}/\ell)$ is the subgroup of matrices $M$ as  in (\ref{eq:matrixnotation}) for which $a_1=a_2=a_3 = 0$. It follows from (\ref{eq:commutator}) that $[U_4(\mathbb{Z}/\ell),U_4(\mathbb{Z}/\ell)]$ is an abelian subgroup of $U_4(\mathbb{Z}/\ell)$, which means that the second derived subgroup of $U_4(\mathbb{Z}/\ell)$ is trivial. The center $Z(U_4(\mathbb{Z}/\ell))$ consists of all matrices $M$ as in (\ref{eq:matrixnotation}) with $a_1=a_2=a_3=u=w=0$. 

Since we assume $\ell>3$, we have that the lower central $\ell$-series of $U_4(\mathbb{Z}/\ell)$ equals its lower central series:
\begin{eqnarray*}
U_4(\mathbb{Z}/\ell)_1=U_4(\mathbb{Z}/\ell)_{1,\ell} &=&  U_4(\mathbb{Z}/\ell),\\
U_4(\mathbb{Z}/\ell)_2=U_4(\mathbb{Z}/\ell)_{2,\ell} &=& [U_4(\mathbb{Z}/\ell),U_4(\mathbb{Z}/\ell)],\\
U_4(\mathbb{Z}/\ell)_3=U_4(\mathbb{Z}/\ell)_{3,\ell} &=& Z(U_4(\mathbb{Z}/\ell)),\\
U_4(\mathbb{Z}/\ell)_i=U_4(\mathbb{Z}/\ell)_{i,\ell} &=& \{\mathrm{identity}\}, \mbox{ for $i\ge 4$}.
\end{eqnarray*}
\end{remark}

\subsection{Triple Massey products for general groups}
\label{ss:tripleMasseygeneral}

We next define triple Massey products for general abstract (resp. profinite) groups. As in \cite{BleherChinburgGillibert2023}, we follow the sign convention in \cite[\S1]{Kraines1996} rather than in \cite{Dwyer1975}.

\begin{definition}
\label{def:tripleMassey}
Let $\Gamma$ be an abstract (resp. profinite) group, and suppose $\Gamma$ acts trivially on $\mathbb{Z}/\ell$.
Let $\chi_1,\chi_2,\chi_3\in \HH^1(\Gamma,\mathbb{Z}/\ell)=\mathrm{Hom}(\Gamma,\mathbb{Z}/\ell)$.  
\begin{itemize}
\item[(a)]
A \emph{defining system} for the triple Massey product $\langle \chi_1,\chi_2,\chi_3\rangle$ is given by a pair  $\kappa_{1,2},\kappa_{2,3}\in\mathrm{Map}( \Gamma ,\mathbb{Z}/\ell)$  such that there is a (continuous) group homomorphism 
$$\begin{array}{crcc}
\tilde{\rho}:&\Gamma &\to&U_4(\mathbb{Z}/\ell)/Z(U_4(\mathbb{Z}/\ell)\\
&\sigma&\mapsto&\left(\begin{array}{cccc}
1&\chi_1(\sigma)& \kappa_{1,2}(\sigma)&*\\
0&1&\chi_2(\sigma)&\kappa_{2,3}(\sigma)\\
0&0&1&\chi_3(\sigma)\\
0&0&0&1\end{array}\right)
\end{array}$$
\item[(b)]
The \emph{triple Massey product} $\langle \chi_1,\chi_2,\chi_3\rangle$ is the subset of $\HH^2(\Gamma,\mathbb{Z}/\ell)$ consisting of the classes of all $2$-cocycles $\nu$ such that there exists a defining system $\{\kappa_{1,2},\kappa_{2,3}\}$ with
$$\nu(\sigma,\tau)=-\chi_1(\sigma)\kappa_{2,3}(\tau) - \kappa_{1,2}(\sigma)\chi_3(\tau)$$
for all $\sigma,\tau\in\Gamma$. In particular, if no defining system $\{\kappa_{1,2},\kappa_{2,3}\}$ exists then $\langle \chi_1,\chi_2,\chi_3\rangle =\emptyset$.
\end{itemize}
\end{definition}

\begin{remark}
\label{rem:tripleMassey}
Let $\Gamma$ and $\chi_1,\chi_2,\chi_3$ be as in Definition \ref{def:tripleMassey}. 
\begin{itemize}
\item[(a)]
The triple Massey product $\langle \chi_1,\chi_2,\chi_3\rangle$ is not empty if and only if $\chi_1 \cup \chi_2 = \chi_2 \cup \chi_3 = 0$ in $\HH^2(\Gamma,\mathbb{Z}/\ell)$. 

\item[(b)]
Suppose $\{\kappa_{1,2},\kappa_{2,3}\}$ is a defining system for $\langle \chi_1,\chi_2,\chi_3\rangle$ and $\tilde{\rho}$ is as in Definition \ref{def:tripleMassey}(a). Then there exists $\rho\in\mathrm{Hom}(\Gamma,U_4(\mathbb{Z}/\ell))$ lifting $\tilde{\rho}$ if and only if the $2$-cocycle $\nu$ in Definition \ref{def:tripleMassey}(b) is the coboundary of a map $\kappa_{1,3}\in\mathrm{Map}(\Gamma , \mathbb{Z}/\ell)$.  

Thus  $\langle \chi_1,\chi_2,\chi_3\rangle$ contains $0$ if and only if there exists $\tilde{\rho}$ as in Definition \ref{def:tripleMassey}(a) that can be lifted to $\rho\in\mathrm{Hom}(\Gamma,U_4(\mathbb{Z}/\ell))$. For more details, see \cite[Thm. 2.4]{Dwyer1975} and also \cite[Lemma 4.2]{MinacTan2015}.

\item[(c)]
Suppose $\kappa=\{\kappa_{1,2},\kappa_{2,3}\}$ is a defining system for $\langle \chi_1,\chi_2,\chi_3\rangle$. Then all defining systems can be obtained from $\kappa$ by adding a morphism $f_{1,2}\in\mathrm{Hom}(\Gamma, \mathbb{Z}/\ell)$ to $\kappa_{1,2}$ and by adding a morphism $f_{2,3}\in\mathrm{Hom}(\Gamma, \mathbb{Z}/\ell)$ to $\kappa_{2,3}$. This means that the 2-cocycle $\nu$ from Definition \ref{def:tripleMassey}(b) gives a single well-defined element $[\nu]$ in the quotient group 
$$\frac{\HH^2(\Gamma,\mathbb{Z}/\ell)}{\HH^1(\Gamma,\mathbb{Z}/\ell)\cup \chi_3 + \chi_1 \cup \HH^1(\Gamma,\mathbb{Z}/\ell)}.$$
\end{itemize}
\end{remark}

\begin{remark}
\label{rem:wretched}
Let $\Omega$ be an abstract (resp. profinite) group, and assume that $\Omega$ acts trivially on $\mathbb{Z}/\ell$. 
Suppose $\bar{\chi}_1,\bar{\chi}_2,\bar{\chi}_3\in \mathrm{Hom}(\Omega,\mathbb{Z}/\ell)$ are characters such that the triple Massey product $\langle\bar{\chi}_1,\bar{\chi}_2,\bar{\chi}_3\rangle$ contains $0$. Let 
$$\rho_\Omega = \rho_\Omega(\bar{\chi}_1,\bar{\chi}_2,\bar{\chi}_3)\in\mathrm{Hom}(\Omega, U_4(\mathbb{Z}/\ell))$$
be a morphism that is associated with $(\bar{\chi}_1,\bar{\chi}_2,\bar{\chi}_3)$ in the sense that the entries above the diagonal are given by $\bar{\chi}_1,\bar{\chi}_2,\bar{\chi}_3$.
For $\sigma,\tau,\gamma\in \Omega$, define 
\begin{equation}
\label{eq:triplecomm}
c_{\sigma,\tau,\gamma}=
\mathrm{det}\left(\begin{array}{rcc}-\bar{\chi}_1(\gamma)&0&\bar{\chi}_3(\gamma)\\ \bar{\chi}_1(\sigma)&\bar{\chi}_2(\sigma)&\bar{\chi}_3(\sigma)\\ \bar{\chi}_1(\tau)&\bar{\chi}_2(\tau)&\bar{\chi}_3(\tau)\end{array}\right).
\end{equation}
It follows from (\ref{eq:commutator}) that $\rho_\Omega([[\sigma,\tau],\gamma])$ is the element in $Z(U_4(\mathbb{Z}/\ell))=U_4(\mathbb{Z}/\ell)_3=U_4(\mathbb{Z}/\ell)_{3,\ell}$ whose entry at the $(1,4)$ position is equal to $c_{\sigma,\tau,\gamma}$ from (\ref{eq:triplecomm}). 
\begin{enumerate}
\item[(a)] We obtain a well-defined (continuous) group homomorphism
$$h : \Omega_3\to\mathbb{Z}/\ell\quad\mbox{with} \quad h([[\sigma,\tau],\gamma]) = c_{\sigma,\tau,\gamma} \quad \mbox{for all $\sigma,\tau,\gamma\in \Omega$.}$$
We note that $h$ only depends on $\bar{\chi}_1,\bar{\chi}_2,\bar{\chi}_3$ but not on $\rho_\Omega$. Moreover, $h$ factors through the quotient group $\Omega_3/\Omega_4$.
\item[(b)] The homomorphism $h$ in part (a) can be 
extended to a well-defined (continuous) group homomorphism
$$h_\ell:\Omega_{3,\ell}\to \mathbb{Z}/\ell.$$ 
This is because $\rho_\Omega$ sends $\Omega_{3,\ell}$ to $U_4(\mathbb{Z}/\ell)_{3,\ell}=Z(U_4(\mathbb{Z}/\ell))$.  Hence, for $\omega \in \Omega_{3,\ell}$,  we just let $h_\ell(\omega)$ be the entry at the $(1,4)$ position of $\rho_\Omega(\omega)$.
Similarly to part (a), $h_\ell$ only depends on $\bar{\chi}_1,\bar{\chi}_2,\bar{\chi}_3$, and $h_\ell$ factors through the quotient group $\Omega_{3,\ell}/\Omega_{4,\ell}$.
\end{enumerate}
\end{remark}

\subsection{Triple Massey products for semidirect products}
\label{ss:tripleMasseysemidirect}

The following result relates the existence and vanishing of triple Massey products for semidirect products $\Omega\rtimes_\varphi G$ to the lower central $\ell$-series of $\Omega$.

\begin{lemma}
\label{lem:cutdowngroup}
Let $\Omega$ and $G$ be both abstract $($resp. profinite$)$ groups, and assume Notation $\ref{not:groupgeneration}$ and Definitions $\ref{def:lowercentral}$ and $\ref{def:tripleMassey}$. Suppose $\varphi\in \mathrm{Hom}(G,\mathrm{Aut}(\Omega))$ is a group action, and assume that $\Omega$ and $G$ act trivially on $\mathbb{Z}/\ell$.
\begin{itemize}
\item[(a)] If $G=G(i,\ell)$ for some $i\ge 2$, then all $\bar{\chi}\in\mathrm{Hom}(\Omega,\mathbb{Z}/\ell)$ extend to characters $\chi\in\mathrm{Hom}(\Omega\rtimes_\varphi G,\mathbb{Z}/\ell)$.

\item[(b)] If $G=G(3,\ell)$ and if $\bar{\chi}_1,\bar{\chi}_2\in \mathrm{Hom}(\Omega,\mathbb{Z}/\ell)$ are characters such that $\bar{\chi}_1\cup\bar{\chi}_2=0$, then $\bar{\chi}_1,\bar{\chi}_2$ extend to characters $\chi_1,\chi_2\in\mathrm{Hom}(\Omega\rtimes_\varphi G,\mathbb{Z}/\ell)$ with the same property. 

In particular, if $\bar{\chi}_1,\bar{\chi}_2,\bar{\chi}_3\in \mathrm{Hom}(\Omega,\mathbb{Z}/\ell)$ are such that the triple Massey product $\langle\bar{\chi}_1,\bar{\chi}_2,\bar{\chi}_3\rangle\neq\emptyset$, then $\bar{\chi}_1,\bar{\chi}_2,\bar{\chi}_3$ extend to characters $\chi_1,\chi_2,\chi_3\in\mathrm{Hom}(\Omega\rtimes_\varphi G,\mathbb{Z}/\ell)$ with the same property.

\item[(c)] If $G=G(4,\ell)$ and if $\bar{\chi}_1,\bar{\chi}_2,\bar{\chi}_3\in \mathrm{Hom}(\Omega,\mathbb{Z}/\ell)$ are characters such that the triple Massey product $\langle\bar{\chi}_1,\bar{\chi}_2,\bar{\chi}_3\rangle$ contains $0$, then $\bar{\chi}_1,\bar{\chi}_2,\bar{\chi}_3$ extend to characters $\chi_1,\chi_2,\chi_3\in\mathrm{Hom}(\Omega\rtimes_\varphi G,\mathbb{Z}/\ell)$ with the same property.
\end{itemize}
\end{lemma}

\begin{proof}
Fix $i\ge 2$. Identifying $\Omega_{i,\ell}$ with the (closed) normal subgroup $\Omega_{i,\ell}\times 1$ of $\Omega \rtimes_{\varphi} G$, we have
\begin{equation}
\label{eq:lame1}
(\Omega \rtimes_{\varphi} G)/\Omega_{i,\ell} = (\Omega/\Omega_{i,\ell}) \rtimes_{\varphi_{i,\ell}} G.
\end{equation}

For part (a), suppose $G=G(i,\ell)$ for some $i\ge 2$. Then (\ref{eq:lame1}) implies
\begin{equation}
\label{eq:lame2}
(\Omega \rtimes_{\varphi} G)/\Omega_{i,\ell} = (\Omega/\Omega_{i,\ell}) \times G.
\end{equation}
Since every character $\bar{\chi}\in\mathrm{Hom}(\Omega,\mathbb{Z}/\ell)$ factors uniquely through $\Omega/\Omega_{i,\ell}$, we can use (\ref{eq:lame2}) to extend $\bar{\chi}$ to a character $\chi\in\mathrm{Hom}(\Omega\rtimes_\varphi G,\mathbb{Z}/\ell)$ by definining
$$\chi(\omega,g) := \bar{\chi}(\omega)$$
for all $\omega\in \Omega$ and $g\in G=G(i,\ell)$.

For part (b), we use that the data of two characters $\bar{\chi}_1,\bar{\chi}_2\in \mathrm{Hom}(\Omega,\mathbb{Z}/\ell)$ such that $\bar{\chi}_1\cup \bar{\chi}_2=0$ is equivalent to the data of a (continuous) group homomorphism 
$$f(\bar{\chi}_1,\bar{\chi}_2):\quad\Omega\to U_3(\Z/\ell)$$
that is associated with $(\bar{\chi}_1,\bar{\chi}_2)$ in the sense that the entries above the diagonal are given by $\bar{\chi}_1,\bar{\chi}_2$.
Since $\ell>3$, and hence $\ell \ne 2$, the third term in the lower central $\ell$-series of $U_3(\Z/\ell)$ is trivial, which implies that the data of $f(\bar{\chi}_1,\bar{\chi}_2)$ is equivalent to the data of a (continuous) group homomorphism 
$$\Omega/\Omega_{3,\ell}\to U_3(\Z/\ell)$$
that is associated with $(\bar{\chi}_1,\bar{\chi}_2)$. Using part (a), and in particular (\ref{eq:lame2}), for $i=3$, part (b) follows. 

Finally, using Definition \ref{def:tripleMassey} and Remark \ref{rem:tripleMassey}(b), together with part (a), and in particular (\ref{eq:lame2}), for $i=4$, part (c) follows. 
\end{proof}

The following result gives a sufficient criterion for the non-vanishing of triple Massey products for semidirect products.

\begin{proposition}
\label{prop:sufficient_general}
Let $\Omega$, $G$ and $\varphi$ be as in Lemma $\ref{lem:cutdowngroup}$, and assume that $\Omega$ and $G$ act trivially on $\mathbb{Z}/\ell$. Suppose $\bar{\chi}_1,\bar{\chi}_2,\bar{\chi}_3\in \mathrm{Hom}(\Omega,\mathbb{Z}/\ell)$ are characters satisfying the following three conditions:
\begin{itemize}
\item[(i)] the triple Massey product $\langle\bar{\chi}_1,\bar{\chi}_2,\bar{\chi}_3\rangle$ is not empty and contains $0$, and
\item[(ii)] none of the $\bar{\chi}_i$ are zero and one of the two sets $\{\bar{\chi}_1,\bar{\chi}_2\}$ or $\{\bar{\chi}_2,\bar{\chi}_3\}$ is $(\mathbb{Z}/\ell)$-linearly independent, and
\item[(iii)] there exist $\Phi\in G(3,\ell)$ and $\omega_0\in \Omega$ such that 
$$\bar{\chi}_1(\omega_0)=\bar{\chi}_2(\omega_0)=\bar{\chi}_3(\omega_0)=0\quad\mbox{and}\quad
h_\ell(\tau_{3,\ell}(\Phi)(\omega_0))\ne 0$$
where $\tau_{3,\ell}$ is as in Remark $\ref{rem:Johnsonhomgeneral}$ and $h_\ell$ is as in Remark $\ref{rem:wretched}(b)$.
\end{itemize}
If there exist characters $\chi_1,\chi_2,\chi_3\in \mathrm{Hom}(\Omega \rtimes_{\varphi} G,\mathbb{Z}/\ell)$ extending $\bar{\chi}_1,\bar{\chi}_2,\bar{\chi}_3$ such that $\langle \chi_1,\chi_2,\chi_3\rangle\neq\emptyset$, then the triple Massey product $\langle \chi_1,\chi_2,\chi_3\rangle $ does not contain $0$. This occurs in particular if $G=G(3,\ell)$.
\end{proposition}

\begin{proof}
By condition (i), the morphism $h_\ell$ from Remark \ref{rem:wretched}(b) is well-defined. Define $\Gamma:=\Omega \rtimes_{\varphi} G$. Let $\{\kappa_{1,2},\kappa_{2,3}\}$ be a defining system for $\langle \chi_1,\chi_2,\chi_3\rangle$, and let 
$$\tilde{\rho}:\,\Gamma\to U_4(\mathbb{Z}/\ell)/Z(U_4(\mathbb{Z}/\ell)$$ 
be the associated (continuous) group homomorphism from Definition \ref{def:tripleMassey}(a). We prove that $\tilde{\rho}$ cannot be lifted to a morphism $\rho\in \mathrm{Hom}\left(\Gamma,U_4(\mathbb{Z}/\ell)\right)$. Suppose to the contrary that such a $\rho$ exists. Let $\rho_\Omega\in \mathrm{Hom}\left(\Omega,U_4(\mathbb{Z}/\ell)\right)$ be the restriction of $\rho$ to $\Omega$. Since $U_4(\mathbb{Z}/\ell)_{4,\ell}$ is trivial, $\rho_\Omega$ factors uniquely through $\Omega/\Omega_{4,\ell}$. Define $M:=\rho(1,\Phi)$. Writing $\Phi.\omega=\varphi(\Phi)(\omega)$ for all $\omega\in\Omega$ and using that $\Phi\in G(3,\ell)$ by condition (iii), we have
\begin{equation}
\label{eq:dumb1}
M\rho_\Omega(\omega)M^{-1}\rho_\Omega(\omega)^{-1} = \rho_\Omega\left((\Phi.\omega)\,\omega^{-1}\right) = \rho_\Omega\left(\tau_{3,\ell}(\Phi)(\omega)\right) \quad\mbox{for all $\omega\in \Omega$}.
\end{equation}
In particular, since $h_\ell(\tau_{3,\ell}(\Phi)(\omega_0))\ne 0$ by condition (iii), we have
\begin{equation}
\label{eq:dumb2}
M\rho_\Omega(\omega_0)M^{-1}\rho_\Omega(\omega_0)^{-1}\ne \mathrm{identity}
\end{equation}
in $U_4(\mathbb{Z}/\ell)$. On the other hand, since $\rho_\Omega\left(\tau_{3,\ell}(\Phi)(\omega)\right)$ lies in $U_4(\mathbb{Z}/\ell)_{3,\ell}=Z(U_4(\mathbb{Z}/\ell))$ for all $\omega\in \Omega$, we obtain from (\ref{eq:dumb1}) that
$$M\rho_\Omega(\omega)M^{-1}\rho_\Omega(\omega)^{-1} \equiv \mathrm{identity} \mod Z(U_4(\mathbb{Z}/\ell)) \quad\mbox{for all $\omega\in \Omega$.}$$
Writing $M=M(a_1,a_2,a_3,u,v,w)$ as in (\ref{eq:matrixnotation}) and using (\ref{eq:commutator}), we get 
$$a_1\bar{\chi}_2(\omega)-a_2\bar{\chi}_1(\omega) = 0 = a_2\bar{\chi}_3(\omega)-a_3\bar{\chi}_2(\omega) \quad \mbox{for all $\omega\in \Omega$. }$$
By condition (ii), this implies $a_1=a_2=a_3=0$, which means that $M\in [ U_4(\mathbb{Z}/\ell),U_4(\mathbb{Z}/\ell)]$. Because $\bar{\chi}_1(\omega_0)=\bar{\chi}_2(\omega_0)=\bar{\chi}_3(\omega_0)=0$ by condition (iii), we have $\rho_\Omega(\omega_0)\in [ U_4(\mathbb{Z}/\ell),U_4(\mathbb{Z}/\ell)]$. Since $[ U_4(\mathbb{Z}/\ell),U_4(\mathbb{Z}/\ell)]$ is commutative, it follows that 
$$M\rho_\Omega(\omega_0)M^{-1}\rho_\Omega(\omega_0)^{-1}=\mathrm{identity}$$
in $U_4(\mathbb{Z}/\ell)$, which contradicts (\ref{eq:dumb2}) and finishes the proof of Propositon \ref{prop:sufficient_general}.
\end{proof}

\begin{remark}
\label{rem:condition_iii}
Conditions (i) and (ii) of Proposition \ref{prop:sufficient_general} imply that $h_\ell$ from Remark \ref{rem:wretched}(b) is surjective.
To show this, it is enough to prove that there exist $\sigma,\tau,\gamma\in \Omega$ such that $c_{\sigma,\tau,\gamma}\ne 0$. 

Suppose first that the three characters $\bar{\chi}_1,\bar{\chi}_2,\bar{\chi}_3$ are $(\mathbb{Z}/\ell)$-linearly independent. Then we can choose $\sigma,\tau,\gamma\in \Omega$ such that the matrix in (\ref{eq:triplecomm}) becomes the identity matrix. Thus $c_{\sigma,\tau,\gamma}=1$. 

Suppose now that only two of the three characters $\bar{\chi}_1,\bar{\chi}_2,\bar{\chi}_3$ are $(\mathbb{Z}/\ell)$-linearly independent. Define
$$D_{12}(\sigma,\tau)=\mathrm{det} \begin{pmatrix} \bar{\chi}_1(\sigma)&\bar{\chi}_2(\sigma)\\ \bar{\chi}_1(\tau)&\bar{\chi}_2(\tau)\end{pmatrix}\quad\mbox{and}\quad
D_{23}(\sigma,\tau)=\mathrm{det} \begin{pmatrix} \bar{\chi}_2(\sigma)&\bar{\chi}_3(\sigma)\\ \bar{\chi}_2(\tau)&\bar{\chi}_3(\tau)\end{pmatrix}.$$
By condition (ii) of Proposition \ref{prop:sufficient_general}, there exist $\sigma,\tau\in \Omega$ such that $D_{12}(\sigma,\tau)\ne 0$ or $D_{23}(\sigma,\tau)\ne 0$. 

If both $D_{12}(\sigma,\tau)\ne 0$ and $D_{23}(\sigma,\tau)\ne 0$, we define
$$\gamma=\sigma^{\bar{\chi}_2(\tau)}\cdot \tau^{-\bar{\chi}_2(\sigma)}.$$
Then $\bar{\chi}_1(\gamma)=D_{12}(\sigma,\tau)$ and $\bar{\chi}_3(\gamma)=-D_{23}(\sigma,\tau)$, and hence $c_{\sigma,\tau,\gamma}=-2\,D_{12}(\sigma,\tau)D_{23}(\sigma,\tau)\ne 0$ since $\ell >3$. 

Suppose next that $D_{12}(\sigma,\tau)= 0$ but $D_{23}(\sigma,\tau)\ne 0$. By condition (ii) of Proposition \ref{prop:sufficient_general}, there exists $\gamma\in\pi_1(\bar{X})$ such that $\bar{\chi}_1(\gamma)\ne 0$. Hence $c_{\sigma,\tau,\gamma}=-\bar{\chi}_1(\gamma)D_{23}(\sigma,\tau)\ne 0$. 
The case when $D_{12}(\sigma,\tau)\ne 0$ but $D_{23}(\sigma,\tau)= 0$ is treated similarly, by using condition (ii) of Proposition \ref{prop:sufficient_general} to find $\gamma\in\Omega$ with $\bar{\chi}_3(\gamma)\ne 0$.
\end{remark}

The following remark reduces Proposition \ref{prop:sufficient_general} to a criterion involving a certain quotient group of the semidirect product $\Omega\rtimes_{\varphi}G$, which is finite when $\Omega$ is finitely generated as abstract (resp. profinite) group. This will be crucial when applying Hilbert's irreducibility theorem. 

\begin{remark}
\label{rem:crucial}
Let $\Omega$, $G$ and $\varphi$ be as in Proposition \ref{prop:sufficient_general}. We obtain a short exact sequence of
abstract (resp. profinite) groups:
\begin{equation}
\label{eq:semidirect1}
1\to \Omega\to \Omega \rtimes_{\varphi} G \to G \to 1.
\end{equation}
Identifying $\Omega_{4,\ell}$ with the (closed) normal subgroup $\Omega_{4,\ell}\times 1$ of $\Omega \rtimes_{\varphi} G$, we see that
\begin{equation}
\label{eq:cutdown1}
(\Omega \rtimes_{\varphi} G)/\Omega_{4,\ell} = (\Omega/\Omega_{4,\ell}) \rtimes_{\varphi_{4,\ell}} G
\end{equation}
where $\varphi_{4,\ell}\in\mathrm{Hom}(G,\mathrm{Aut}(\Omega/\Omega_{4,\ell}))$ is as in Definition \ref{def:lowercentral}(c). Since $G(4,\ell)=\mathrm{kernel}(\varphi_{4,\ell})$, it follows that $\varphi_{4,\ell}$ factors uniquely through $G/G(4,\ell)$. We define
\begin{equation}
\label{eq:cutcut}
\overline{\Omega}:=\Omega/\Omega_{4,\ell},\quad \overline{G}:=G/G(4,\ell),\quad\mbox{and}\quad \overline{\varphi}:=\varphi_{4,\ell}.
\end{equation}
Using (\ref{eq:cutdown1}) and (\ref{eq:cutcut}), the short exact sequence (\ref{eq:semidirect1}) induces a short exact sequence of abstract (resp. profinite) groups:
\begin{equation}
\label{eq:semidirect2}
1\to \overline{\Omega}\to \overline{\Omega} \rtimes_{\overline{\varphi}} \overline{G} \to \overline{G} \to 1.
\end{equation}
If  $\Omega$ is \emph{finitely generated} as abstract (resp. profinite) group, then Remark \ref{rem:lowercentral}(b) shows that (\ref{eq:semidirect2}) is a short exact sequence of finite groups. 

Since $U_4(\mathbb{Z}/\ell)_{4,\ell}$ is trivial and since $h_\ell$ factors through the quotient group $\Omega_{3,\ell}/\Omega_{4,\ell}$ and  $\tau_{3,\ell}$ factors through the quotient group $G(3,\ell)/G(4,\ell)$, the data of three characters $\bar{\chi}_1,\bar{\chi}_2,\bar{\chi}_3\in \mathrm{Hom}(\Omega,\mathbb{Z}/\ell)$ satisfying conditions (i), (ii) and (iii) of Proposition \ref{prop:sufficient_general} is equivalent to the data of three characters $\bar{\chi}_1,\bar{\chi}_2,\bar{\chi}_3\in \mathrm{Hom}(\overline{\Omega},\mathbb{Z}/\ell)$ satisfying these conditions for $\overline{\Omega}$, $\overline{G}$ and $\overline{\varphi}$ instead of $\Omega$, $G$ and $\varphi$.
\end{remark}

The next two remarks discuss profinite completions.

\begin{remark}
\label{rem:profcompl1}
For any group $\Gamma$ (resp. any group homomorphism $f:\Gamma_1\to\Gamma_2$), we denote by $\widehat{\Gamma}$ (resp. $\widehat{f}:\widehat{\Gamma_1}\to \widehat{\Gamma_2}$) its profinite completion. For any subgroup $H\le \Gamma$, we denote by $H^{\mathrm{cl}}$ the \emph{closure} of the image of $H$ in $\widehat{\Gamma}$. 

\begin{itemize}
\item[(a)] 
If $H$ is a finite-index subgroup of $\Gamma$ then $H^{\mathrm{cl}} \cong \widehat{H}$. 

\item[(b)]
If $\Gamma=\Omega\rtimes_{\varphi}G$ is a semidirect product and $\Omega$ is \emph{finitely generated}, then the closure $\Omega^{\mathrm{cl}}$ of the image of $\Omega$ in $\widehat{\Gamma}$ satisfies $\Omega^{\mathrm{cl}} \cong \widehat{\Omega}$.

One way to see this is that $\Omega$ has a cofinal sequence $\Omega_1 > \Omega_2 > \cdots$ of finite-index characteristic subgroups. This implies that  $\{\Omega_i\rtimes_{\varphi} \Gamma\}_{i\ge 1}$ is a sequence of normal finite-index subgroups of $\Gamma$ such that for any normal finite-index subgroup $N$ of $\Omega$ there exists $i_0\ge 1$ satisfying $(\Omega_{i_0}\rtimes_{\varphi} \Gamma) \cap \Omega <N$.
\end{itemize}
\end{remark}

\begin{remark}
\label{rem:profcompl2}
Let $\Omega$, $G$ and $\varphi$ be as in Proposition \ref{prop:sufficient_general}. Suppose $\Omega$ and $G$ are abstract groups, and that $\Omega$ is \emph{finitely generated}. Using Remark \ref{rem:profcompl1}(b), we obtain a commutative diagram 
\begin{equation}
\label{eq:complete1}
\xymatrix{
1\ar[r] & \Omega \ar[r]\ar[d] & \Omega \rtimes_{\varphi} G \ar[r]\ar[d] & G \ar[r]\ar[d] & 1\\
1\ar[r]& \widehat{\Omega} \ar[r] & \widehat{\Omega} \rtimes_{\widehat{\varphi}} \widehat{G} \ar[r] & \widehat{G} \ar[r] & 1}
\end{equation}
with exact rows. In particular, we have an isomorphism of profinite groups
$$\widehat{\Omega \rtimes_{\varphi} G} \cong \widehat{\Omega} \rtimes_{\widehat{\varphi}}\widehat{G}.$$
Using Remark \ref{rem:lowercentral}(b) and Remark \ref{rem:profcompl1}(a), we obtain the following identifications of finite groups:
$$\Omega/\Omega_{4,\ell}=\widehat{\Omega}/\widehat{\Omega}_{4,\ell}
\quad\mbox{and}\quad G/G(4,\ell)=\widehat{G}/\widehat{G}(4,\ell).$$
Using the notation from Remark \ref{rem:crucial}, we can also identify
$$\overline{\Omega} \rtimes_{\overline{\varphi}} \overline{G} = (\widehat{\Omega}/\widehat{\Omega}_{4,\ell})\rtimes_{\widehat{\varphi}_{4,\ell}}(\widehat{G}/\widehat{G}(4,\ell)).$$
Therefore, Remark \ref{rem:crucial} implies that the data of three characters $\bar{\chi}_1,\bar{\chi}_2,\bar{\chi}_3\in \mathrm{Hom}(\Omega,\mathbb{Z}/\ell)$ satisfying conditions (i), (ii) and (iii) of Proposition \ref{prop:sufficient_general} is equivalent to the data of three characters $\bar{\chi}_1,\bar{\chi}_2,\bar{\chi}_3\in \mathrm{Hom}(\widehat{\Omega},\mathbb{Z}/\ell)$ satisfying these conditions for $\widehat{\Omega}$, $\widehat{G}$ and $\widehat{\varphi}$ instead of $\Omega$, $G$ and $\varphi$.
\end{remark}

\subsection{Triple Massey products for smooth projective curves}
\label{ss:tripleMasseySmooth} 

Let $F$ be a field with a fixed separable closure $\bar{F}$ inside a fixed algebraic closure $F^{\mathrm{alg}}$, and let $\ell$ be a prime number that is invertible in $F$. Let $X$ be a smooth projective geometrically irreducible curve over $F$, and define 
$$\quad \bar{X}:=X\otimes_F\bar{F}.$$ 
Let $\eta$ be a geometric point of $X$, i.e. a point with values in $\bar{F}$, which we also view as a geometric point of $\bar{X}$. To simplify notation, we will often denote the \'etale fundamental groups by
$$\pi_1(X):=\pi_1(X,\eta) \quad \text{and}\quad  \pi_1(\bar{X}):=\pi_1(\bar{X},\eta).$$
We  have a natural isomorphism of first cohomology groups
$$\HH^1(X,\mathbb{Z}/\ell) \cong \HH^1(\pi_1(X),\mathbb{Z}/\ell),$$
where, for $i\ge 0$, $\HH^i(X,\mathbb{Z}/\ell)$ denotes the $i$-th \'etale cohomology group of $X$ with coefficients in $\mathbb{Z}/\ell$ and $\HH^i(\pi_1(X),\mathbb{Z}/\ell)$ denotes the $i$-th continuous cohomology group of $\pi_1(X)$ with coefficients in $\mathbb{Z}/\ell$. Since $\pi_1(X)$ acts trivially on $\mathbb{Z}/\ell$, we can identify $\HH^1(\pi_1(X),\mathbb{Z}/\ell)=\mathrm{Hom}(\pi_1(X),\mathbb{Z}/\ell)$.
For higher cohomology groups, we have the following result from \cite[\S2.1.2]{AchingerThesis} (see also \cite[\S3]{Achinger2015}):

\begin{proposition}
\label{prop:Achinger}
If $\bar{X}$ is not isomorphic to $\mathbb{P}^1_{\bar{F}}$ then there is a natural isomorphism
$$\HH^i(X,\mathbb{Z}/\ell) \cong \HH^i(\pi_1(X),\mathbb{Z}/\ell)$$
for all $i\ge 1$.
\end{proposition}

For the remainder of the paper we assume that $\bar{X}$ is not isomorphic to $\mathbb{P}^1_{\bar{F}}$, and we identify $\HH^i(X,\mathbb{Z}/\ell) = \HH^i(\pi_1(X),\mathbb{Z}/\ell)$ for all $i\ge 1$. 

\begin{nota}
\label{not:rationalpoint}
Let $G_F=\mathrm{Gal}(\bar{F}/F)$. We have a short exact sequence of profinite groups and continuous group homomorphisms
$$1\to \pi_1(\bar{X}) \to \pi_1(X) \to G_F\to 1.$$
Assume $X(F) \neq \emptyset$; then the choice of a point $\mathfrak{p}\in X(F)$, viewed as a morphism $\mathfrak{p}:\Spec(F)\to X$, gives a map $\mathfrak{p}_*:G_F\to \pi_1(X)$ which is a section of this exact sequence. In other words,
$$\pi_1(X)=\pi_1(\bar{X})\rtimes_{\varphi} G_F$$
where $\varphi : G_F \to\mathrm{Aut}(\pi_1(\bar{X}))$ sends $\Phi\in G_F$ to the automorphism of $\pi_1(\bar{X})$ given by conjugation with $\mathfrak{p}_*(\Phi)$.
For $\Phi\in G_F$ and $\xi\in \pi_1(\bar{X})$, we use the notation 
$$\Phi.\xi=\varphi(\Phi)(\xi)= \mathfrak{p}_*(\Phi)\,\xi\,\mathfrak{p}_*(\Phi)^{-1}.$$ 

As in Definition \ref{def:lowercentral}(c), for $i\ge 1$, $G_F(i,\ell)$ denotes the closed normal subgroup of $G_F$ consisting of all $\Phi\in G_F$ that act trivially on $\pi_1(\bar{X})/\pi_1(\bar{X})_{i,\ell}$.
\end{nota}

Using Notation \ref{not:rationalpoint}, the following result shows that conditions (i) and (ii) of Proposition \ref{prop:sufficient_general} are necessary for the non-vanishing of triple Massey products for $\pi_1(X)$ when $\Omega=\pi_1(\bar{X})$ and $G=G_F$ provided the $\ell$-torsion $\mathrm{Pic}(\bar{X})[\ell]$ is defined over $F$.

\begin{proposition}
\label{prop:necessary_curves}
Let $\chi_1,\chi_2,\chi_3\in \HH^1(X,\mathbb{Z}/\ell)=\mathrm{Hom}(\pi_1(X),\mathbb{Z}/\ell)$, and let $\bar{\chi}_1,\bar{\chi}_2,\bar{\chi}_3$ denote their restrictions to $\HH^1(\bar{X},\mathbb{Z}/\ell)=\mathrm{Hom}(\pi_1(\bar{X}),\mathbb{Z}/\ell)$. Assume the triple Massey product $\langle\chi_1,\chi_2,\chi_3\rangle$ is not empty. Then
\begin{itemize}
\item[(i)] the triple Massey product $\langle \bar{\chi}_1,\bar{\chi}_2,\bar{\chi}_3 \rangle$ always contains zero. 
\end{itemize}
Suppose additionally that $\mathrm{Pic}(\bar{X})[\ell]$ is defined over $F$ and $\langle \chi_1,\chi_2,\chi_3\rangle$ does not contain zero. Then
\begin{itemize}
\item[(ii)]  none of the $\bar{\chi}_i$ are zero and one of the two sets $\{\bar{\chi}_1,\bar{\chi}_2\}$ or $\{\bar{\chi}_2,\bar{\chi}_3\}$ is $(\mathbb{Z}/\ell)$-linearly independent.
\end{itemize}
\end{proposition}

\begin{proof}
Part (i) follows from \cite[Prop. 3.1]{BleherChinburgGillibert2023}.

To prove part (ii), suppose that $\mathrm{Pic}(\bar{X})[\ell]$ is defined over $F$ and $\langle \chi_1,\chi_2,\chi_3\rangle$ does not contain zero. By \cite[Prop. 4.1(a)]{BleherChinburgGillibert2023}, none of the $\bar{\chi}_i$ are zero. Suppose each of $\{\bar{\chi}_1,\bar{\chi}_2\}$ and $\{\bar{\chi}_2,\bar{\chi}_3\}$ are $(\mathbb{Z}/\ell)$-linearly dependent.  This means there exist $a_1,a_2 \in (\mathbb{Z}/\ell)^\times$ with $\bar{\chi}_1=a_1\bar{\chi}_2$ and $\bar{\chi}_2=a_2\bar{\chi}_3$. Hence there exist $\psi_1,\psi_2\in \HH^1(F,\mathbb{Z}/\ell) = \mathrm{Hom}(G_F,\mathbb{Z}/\ell)$ such that
$$\chi_1 = a_1 \chi_2 + \psi_1 \quad\mbox{and}\quad \chi_2 = a_2 \chi_3 + \psi_2.$$
Since $\chi_1\cup \chi_2 = 0$, and $\chi_2 \cup \chi_2 = 0$ because $\ell> 3$, this implies $\psi_1\cup\chi_2=0$. But since $\bar{\chi}_2\ne 0$ and $\bar{\psi}_1=0$, this is by \cite[Lemma 2.6]{BleherChinburgGillibert2023} only possible if $\psi_1=0$. Similarly,  $\psi_2=0$. Thus there exist $a_1,a_2 \in (\mathbb{Z}/\ell)^\times$ with $\chi_1=a_1\chi_2$ and $\chi_2=a_2\chi_3$. But this implies by \cite[Prop. 4.1(b)]{BleherChinburgGillibert2023} that $\ell \le 3$, which contradicts our assumption that $\ell > 3$. Therefore, $\{\bar{\chi}_1,\bar{\chi}_2\}$ or $\{\bar{\chi}_2,\bar{\chi}_3\}$ must be $(\mathbb{Z}/\ell)$-linearly independent, which finishes the proof of part (ii).
\end{proof}


\section{Mumford's construction and Johnson's homomorphisms}
\label{s:topology}

\subsection{Mumford's construction of hyperelliptic curves}
\label{ss:hyperell}

We begin by recalling some constructions of Mumford in \cite[p. 3.124 - 3.134]{Mumford}
concerning moduli spaces of hyperelliptic curves of genus $g \ge 1$.

 Let $\mathcal{H}_g^{(2)}$ the open subset of $(\mathbb{C} - \{0,1\})^{2g -1}$ consisting of all points $(a_2,\ldots,a_{2g})$ such that the $a_i$ are distinct.  Thus $\mathcal{H}_g^{(2)}$ is the set of ordered subsets of size $2g - 1$ in $\mathbb{C}  - \{0,1\}$.  Over $\mathcal{H}_g^{(2)}$ we have a family $\mathcal{C}$ of smooth projective hyperelliptic curves such that the fiber  $\mathcal{C}_a$ of $\mathcal{C}$ over $a = (a_2,\ldots,a_{2g})$ has affine equation
$$y^2 = x \cdot (x - 1) \cdot (x - a_2) \cdots (x - a_{2g}).$$
Note that there will be a unique point over $x = \infty$ on this curve, and we must choose a smooth model of the family $\mathcal{C}$ over $x = \infty$.
The projection $(x,y)  \to x$ makes $\mathcal{C}_a$ a double cover of $\mathbb{P}^1(\mathbb{C})$ that is branched at $\{0,1,a_2,\ldots,a_{2g},\infty\}$.  We define $\infty(a)$ to be the unique point of $\mathcal{C}_a$ over $\infty \in \mathbb{P}^1(\mathbb{C})$.

Fix the base point $b = (2,3,\ldots,2g) \in \mathcal{H}_g^{(2)}$.  Mumford constructs the universal cover $\hat{\mathcal{H}}_g$ of $\mathcal{H}_g^{(2)}$ as the space of homotopy classes of continuous maps 
$$\phi:[0,1] \to (\mathbb{C} - \{0,1\})^{2g -1} - \cup_{i < j} \Delta_{i,j}$$
such that $\phi(0) = b$, where a homotopy between two such maps must fix the starting and ending points.
The map
$$\hat{\mathcal{H}}_g \to \mathcal{H}_g^{(2)}$$
defined by $\phi \to \phi(1)$ makes $\hat{\mathcal{H}}_g$ a covering space of $\mathcal{H}_g^{(2)}$ with covering group $G = \pi_1^{\mathrm{top}}(\mathcal{H}_g^{(2)},b)$.  

We identify $G$ with the so-called normalized pure braid group on $2g+2$ strands in $\mathbb{P}^1(\mathbb{C})$ based at $B_0 = (0,1,2,3,\ldots,2g,\infty)$.  Here braids that start at $B_0$ are pure if each strand begins and ends at the same point. A pure braid is normalized if the strands at $0$, $1$ and $\infty$ are constant, and if no strand wraps around the strand at $\infty$.  We can thus also view $G$ as the normalized pure braid group associated to  $2g+1$ strands in $\mathbb{C}$;  see \cite[p. 3.126]{Mumford}.

 Let $\mathcal{G}$ be the group of orientation preserving homeomorphisms $\phi:\mathbb{P}^1(\mathbb{C}) \to \mathbb{P}^1(\mathbb{C})$ such that $\phi(0) = 0$, $\phi(1) = 1$ and $\phi(\infty) = \infty$.  The topology of $\mathcal{G}$ is the compact open topology.  Define $K_g$ to be the subgroup of $\phi \in \mathcal{G}$ such that $\phi(i) = i$ for $i = 0,1, \ldots, 2g, \infty$.  The map $\pi:\mathcal{G} \to \mathcal{H}_g^{(2)}$ defined by 
$$\pi(\phi)  = (\phi(2),\ldots,\phi(2g))$$ then gives
a bijection 
\begin{equation}
\label{eq:bije}
\pi:\mathcal{G}/K_g \to \mathcal{H}_g^{(2)}.
\end{equation}

Mumford shows in \cite[Lemma 8.11]{Mumford} that near every point $a \in \mathcal{H}_g^{(2)}$ there is a local section of the projection $\pi:\mathcal{G} \to \mathcal{H}_g^{(2)}$.  These local sections amount to making a continuous choice for $a'$  in a neighborhood of $a \in \mathcal{H}_g^{(2)}$ of a homeomorphism 
$\phi_{a'} \in \mathcal{G}$ such that $\pi(\phi_{a'}) = a'$.  Thus $\pi$ has the homotopy lifting property.  In particular, we can lift paths $\gamma:[0,1] \to \mathcal{H}_g^{(2)}$ such that $\gamma(0) = b = (2,3,\ldots,2g)$ to paths $\phi:[0,1] \to \mathcal{G}$ such that $\phi(0)$ is the identity element of $\mathcal{G}$.  We will view such lifts in the following way.  As $r$ varies from $0$ to $1$, $\gamma(r) \in \mathcal{H}_g^{(2)}$ traces out a continuous path of locations for the branch points in $\mathbb{P}^1(\mathbb{C})$ for hyperelliptic curves $\mathcal{C}_{\gamma(r)}$ in the family $\mathcal{C} \to \mathcal{H}_g^{(2)}$.  The path $\phi:[0,1] \to \mathcal{G}$ then produces a way of dragging $\mathbb{P}^1(\mathbb{C})$ around over this path in such a way that the initial set of branch points $B_0 = (0,1,2,\ldots,2g,\infty)$ associated to $b =  \gamma(0) = \pi(\phi(0)) = (2,3,\ldots,2g)$ are carried to $(0,1,a_2,\ldots,a_{2g},\infty)$ when $\pi(\phi(1)) = (a_2,\ldots,a_{2g}) = \gamma(1)$.  Since $\phi(0)$ is by assumption the identity element of $\mathcal{G}$, we can lift $\phi(0)$ to the identity automorphism $\tilde{\phi}(0)$ of the
fiber $\mathcal{C}_b$ of the family $\mathcal{C} \to \mathcal{H}_g^{(2)}$.  There is then a unique way to lift $\phi(r) \in \mathcal{G}$ for $0 < r \le 1$ to a homeomorphism $\tilde{\phi}(r):\mathcal{C}_{b} \to \mathcal{C}_{\pi(\phi(r))} = \mathcal{C}_{\gamma(r)}$ such that the function $r \mapsto \tilde{\phi}(r)$ is continuous over $0 \le r \le 1$.  Thus $\phi:[0,1] \to \mathcal{G}$ also produces a continuous way to drag around the fibers 
$\mathcal{C}_{\gamma(r)}$ of the restriction of the family $\mathcal{C} \to \mathcal{H}_g^{(2)}$ over the
path $\gamma:[0,1] \to \mathcal{H}_g^{(2)}$.

Since local sections of $\mathcal{G} \to \mathcal{H}_g^{(2)}$ are not unique, there are many choices of how to lift a path $\gamma:[0,1] \to \mathcal{H}_g^{(2)}$ with $\gamma(0) = b$ to a path $\phi:[0,1] \to \mathcal{G}$ with $\phi(0)$ equal to the identity in $\mathcal{G}$.   Let $[\gamma]$ be the element of the universal cover $\hat{\mathcal{H}}_g$ that is determined by the homotopy class of $\gamma$.  We lift $\gamma$ as above with $\phi(0) \in \mathcal{G}$ the identity element of $\mathcal{G}$.  Mumford shows in 
\cite[p. 3.128 - 3.129]{Mumford} that because of the homotopy lifting property of $\pi:\mathcal{G} \to \mathcal{H}_g^{(2)}$, $[\gamma]$ determines $\phi(1) \in \mathcal{G}$  up to right multiplication by an element of the subgroup $K_g^0$ of $K_g$ that is the path connected component of the identity of 
$K_g$.   Here  $\beta \in K_g$ lies in $K_g^0$ if there is a continuous $\tau:[0,1] \to K_g$ with $\tau(0) = $ identity and $\tau(1) = \beta$. 
This gives a map
\begin{equation}
\label{eq:sigmamap} \sigma: \hat{\mathcal{H}}_g \to \mathcal{G}/K_g^0
\end{equation}
defined by $[\gamma] \mapsto \phi(1)K_g^0$.

In more colloquial terms, the homotopy class $[\gamma]$ determines the final homeomorphism
$\phi(1) :\mathbb{P}^1(\mathbb{C}) \to \mathbb{P}^1(\mathbb{C})$ up to precomposition by $\tau(1)$ for
a continuous family of homeomorphisms $\tau:[0,1] \to K_g$ such that $\tau(0)$ is the identity map.  Here by the definition of $K_g$, each
$\tau(r):\mathbb{P}^1(\mathbb{C}) \to \mathbb{P}^1(\mathbb{C})$ fixes $0$, $1$, $\infty$, and each coordinate of the base point
$b = (2,\ldots,2g) \in \mathcal{H}_g^{(2)}$.  As above, we lift the identity map $\tau(0)$ to the identity map $\tilde{\tau}(0)$ of the fiber  $\mathcal{C}_b$.  Then for  $r \in [0,1]$, there  will be a unique lift of $\tau(r) \in K_g$ to an automorphism $\tilde{\tau}(r)$ of  $\mathcal{C}_b$ such that $r \mapsto \tilde{\tau}(r)$ is continuous.
Thus $[\gamma]$ determines the homeomorphism
$\tilde{\phi}(1):\mathcal{C}_{b} \to \mathcal{C}_{\pi(\phi(1))} = \mathcal{C}_{\gamma(1)}$ 
of hyperelliptic curves up to precomposition with a homeomorphism of $\mathcal{C}_{b}$ that can be connected to the identity homeomorphism by a family of homeomorphisms of $\mathcal{C}_{b}$.  

Suppose now that $\gamma: [0,1] \to \mathcal{H}_g^{(2)}$ is actually a loop based at $b$, so that
$[\gamma]$ is an element of $\pi_1^{\mathrm{top}}(\mathcal{H}_g^{(2)},b)$.  Then $\phi(1)$ (resp. $\tilde{\phi}(1)$) is a homeomorphism of
$\mathbb{P}^1(\mathbb{C}) - \{0,1,2,\ldots,2g,\infty\}$ (resp. $\mathcal{C}_{b}$) that is determined up to isotopy.  Furthermore, $\tilde{\phi}(1)$ sends the point $\infty(b) \in \mathcal{C}_b$ to itself, since $\infty(b)$ is the unique point of $\mathcal{C}_b$ over $\infty \in \mathbb{P}^1(\mathbb{C})$ and $\phi(1) \in \mathcal{G}$ fixes $\infty$.  Let $\mathrm{Mod}\left(\mathbb{P}^1(\mathbb{C}) - \{0,1,2,\ldots,2g,\infty\}\right)$ be the mapping class group of $\mathbb{P}^1(\mathbb{C}) - \{0,1,2,\ldots,2g,\infty\}$.  Let 
$\mathrm{Mod}\left(\mathcal{C}_{b}, \infty(b)\right)$ be the mapping class group of $\mathcal{C}_b$ with marked point $\infty(b)$, i.e. the group of homeomorphisms of  $\mathcal{C}_b$ that fix $\infty(b)$ up to isotopy.  We conclude that $\phi(1)$ (resp. $\tilde{\phi}(1)$) defines an element $z([\gamma])$ of $\mathrm{Mod}\left(\mathbb{P}^1(\mathbb{C}) - \{0,1,2,\ldots,2g,\infty\}\right)$ (resp. $\tilde{z}([\gamma])$ of $\mathrm{Mod}\left(\mathcal{C}_{b}, \infty(b)\right)$).  We obtain group homomorphisms
\begin{eqnarray}
\label{eq:mappingclassB}
z:\quad \pi_1^{\mathrm{top}}(\mathcal{H}_g^{(2)},b)&\to&\mathrm{Mod}\left(\mathbb{P}^1(\mathbb{C}) - \{0,1,2,\ldots,2g,\infty\}\right),\quad\mbox{and}\\
\label{eq:mappingclassC}
\tilde{z}:\quad \pi_1^{\mathrm{top}}(\mathcal{H}_g^{(2)},b)&\to&\mathrm{Mod}\left(\mathcal{C}_{b},\infty(b)\right).
\end{eqnarray}
With more work, one could refine these homomorphisms to have images in the mapping class groups associated to homeomorphisms of $\mathbb{P}^1(\mathbb{C}) - \{0,1,2,\ldots,2g,\infty\}$ (resp. $\mathcal{C}_b$) that fix all points in a small closed disc of positive radius around $\infty \in \mathbb{P}^1(\mathbb{C})$ (resp. around $\infty(b)$  on $\mathcal{C}_b$).  However, we will not need such a refinement, since these more refined mapping class groups map naturally to the ones above, and for our purposes constructing elements of $\mathrm{Mod}\left(\mathcal{C}_{b},\infty(b)\right)$ will suffice. 

Restricting the map $\sigma$ in (\ref{eq:sigmamap}) to the homotopy classes of loops $\gamma:[0,1] \to \mathcal{H}_g^{(2)} $ based at $b$ gives a homomorphism
 \begin{equation}
 \label{eq:sigmalittle}
 \sigma_*:G = \pi_1^{\mathrm{top}}(\mathcal{H}_g^{(2)}, b) \to K_g/K_g^0.
 \end{equation}
 Here $\pi_1^{\mathrm{top}}(\mathcal{H}_g^{(2)}, b) $ is the covering group of 
 $\hat{\mathcal{H}}_g \to \mathcal{H}_g^{(2)} = \mathcal{G}/K_g$
 and the natural map $\mathcal{G}/K_g \to \mathcal{G}/K_g^0$ has covering group $K_g/K_g^0$ acting on the right.
 Mumford shows that the map $\sigma: \hat{\mathcal{H}}_g \to \mathcal{G}/K_g^0$ in (\ref{eq:sigmamap})
 is equivariant for the homomorphism $\sigma_*$ in (\ref{eq:sigmalittle}).  This gives a commutative square
 $$ \xymatrix @C.3pc {
\hat{\mathcal{H}}_g\ar[d]&\to&\mathcal{G}/K_g^0 \ar[d]\\
\mathcal{H}_g^{(2)}&\to& \mathcal{G}/K_g
}$$
in which the action of $G$ on the upper left corner over the bottom left corner is consistent
with the action of $K_g/K_g^0$ on the upper right corner over the bottom right corner via 
$\sigma_*:G = \pi_1^{\mathrm{top}}(\mathcal{H}_g^{(2)}, b) \to K_g/K_g^0$.  
 
Suppose now that $S$ is a simple closed curve on $\mathbb{P}^1(\mathbb{C}) - \{0,1,2,\ldots,2g,\infty\}$.  One can form a Dehn twist $\delta(S) \in \mathcal{G}$  in the following way.  Choose a collar $U$ around $S$ in $\mathbb{P}^1(\mathbb{C}) - \{0,1,2,\ldots,2g,\infty\}$, so that $U$ is homeomorphic to an annulus with two boundary components that are circles.  Considering $S$ as a parameterized curve, hold the left boundary of $U$ fixed and twist the right boundary of $U$ through one full rotation while extending this twist continuously to the interior of $U$.  This produces a homeomorphism $\delta(S)\in \mathcal{G}$ that is the identity outside of $U$.  

Mumford gives the following recipe for producing a path $\phi:[0,1] \to \mathcal{G}$ such that $\phi(0)$ is the identity of $\mathcal{G}$ and $\phi(1) = \delta(S)$.  Let $x$ be a point of $S$.  Since $S$ is contractible in $\mathbb{P}^1(\mathbb{C})  - \{\infty\}$, we can find a homotopy $\Phi:[0,1] \times [0,1] \to \mathbb{P}^1(\mathbb{C}) - \{\infty\}$ such that $\Phi(0,t) = x$ for $t \in [0,1]$, $\Phi(1,t)$ traces the curve $S$ as $t$ ranges from $0$ to $1$ with $\Phi(s,0) = x = \Phi(s,1)$ for all $s \in [0,1]$.  As $s$ varies from $1$ down to $0$, we view $\Phi(s \times [0,1])$ as a loop $S_s$ in $\mathbb{P}^1(\mathbb{C})$ that is based at $x$ with $S_1 = S$ and $S_s$ shrinking to the constant loop  $S_0$ at $x$ as $s \to 0$.   We can further arrange that for all $s \in [0,1]$, there exists at most one element of $\{0,1,\ldots,2g\}$ that lies on  $S_s$.   

Mumford now begins with the Dehn twist $\delta(S)$ associated to  $S_1 = S$ and deforms the annulus $U = U_1 $ around $S_1$ to produce a varying family of annuli $U_s$ around $S_s$ with the property that at most one point of 
$\{0,1,2,\ldots,2g\}$  is in any $U_s$.  We now perform continuously varying Dehn twists on the 
$U_s$ as $s$ decreases from $1$ to $0$.  For intervals of $s$ such that $U_s$ contains no point of $\{0,1,2,\ldots, 2g\}$ this moves none of the points of $\{0,1,2,\ldots,2g,\infty\}$, so we get a continuously varying family $\phi_s$ of elements of $K_g \subset \mathcal{G}$ parameterized by $s$ in such intervals.  When one edge of $U_s$ touches a point $i \in \{2,\ldots,2g\}$, the point $i$ starts to be moved by the Dehn twists associated to decreasing $s$.  The corresponding $\phi_s \in \mathcal{G}$ now have $\phi_s(i) \ne i$, but
$\phi_s(0) = 0$, $\phi_s(1) = 1$ and $\phi_s(\infty) = \infty$.  Eventually $s$ decreases enough to where $\phi_s(i)$ now touches the other boundary of $U_s$ and $\phi_s(i)$ is brought back to $\phi_s(i) = i$.
A physical analogy is that the moving family of Dehn twists is a cyclone moving across the landscape, with a cow at position $i$ being picked up, moved around and then dropped back down to position $i$ after going once around the cyclone.  The situation when $U_s$ touches $i = 0$ or $i = 1$ is  more subtle, since we want to produce elements of $\mathcal{G}$, and these fix $0$, $1$ and $\infty$.  To deal with $i = 0$ being moved, for example, one applies fractional linear transformations of $\mathbb{P}^1(\mathbb{C})$ which bring where $i = 0$ would be moved by the Dehn twists back to  $i = 0$ while fixing $1$ and $\infty$.   Note that these fractional linear transformations will then produce $\phi_s \in \mathcal{G}$ that move points outside of $U_s$.  The  procedure when $i  = 1$ would be moved by a Dehn twist is similar.  Mumford refers to this as ``putting $0$, $1$ and $\infty$ back by a unique projectivity."  A more prosaic description is that the cows located at $i = 0$ and $i = 1$ are particularly stubborn and do not want to move.  So they force the rest of the world to do the moving via fractional linear transformations when the Dehn cyclone reaches them.

We summarize the conclusions of this construction in the following lemma:
 
 \begin{lemma}
 \label{lem:nicepic}  
 The Dehn twist $\delta(S) \in \mathcal{G}$ of a simple closed curve $S$ on
 $Y= \mathbb{P}^1(\mathbb{C}) - \{0,1,2,\ldots,2g,\infty\}$ arises as $\phi(1)$ for a path $\phi:[0,1] \to \mathcal{G}$ with $\phi(0)$ the identity of $\mathcal{G}$.
 The image of $\phi$ under the projection $\pi:\mathcal{G} \to \mathcal{H}_g^{(2)}$ is a loop $\alpha:[0,1] \to \mathcal{H}_g^{(2)}$ based at $b = \alpha(0) = \alpha(1) = (2,3,\ldots,2g)$.  
 \begin{itemize}
 \item[(a)] The element $z([\alpha])$ of the mapping class group of $Y$ that is associated to the class $[\alpha] \in \pi_1^{\mathrm{top}}(\mathcal{H}_g^{(2)},b)$ is represented by $\delta(S) = \phi(1)$.  Thus the Dehn twist $\delta(S)$ is induced by the normalized pure braid associated to $[\alpha]$. 
\item[(b)]
 For $t \in [0,1]$, write $\alpha(t) = (\alpha(t)_2,\ldots,\alpha(t)_{2g}) = \pi(\phi(t))$. The element $\phi(t) \in \mathcal{G}$ defines  a homeomorphism between $Y = Y_b$ and $Y_{\alpha(t)} = \mathbb{P}^1(\mathbb{C}) - \{0,1,\alpha(t)_2,\ldots,\alpha(t)_{2g},\infty\}$ which varies continuously with $t$.  After choosing $\tilde{\phi}(0)$ to be the identity on the fiber $\mathcal{C}_b$, these homeomorphisms lift uniquely and continuously to homeomorphisms
 $\tilde{\phi}(t)$ between the fiber $\mathcal{C}_b$ of the family of hyperelliptic curves $\mathcal{C} \to \mathcal{H}_g^{(2)}$ and the fiber $\mathcal{C}_{\alpha(t)}$.  Furthermore $\tilde{\phi}(t)$ carries $\infty(b)$ to $\infty(\alpha(t))$.  The canonical lift $\tilde{\delta}(S)$ of the Dehn twist to $\mathcal{C}_b$ is represented by $\tilde{\phi}(1)$.
 \end{itemize}
 \end{lemma}
 
 By \cite[p. 249]{FarbMargalit2012}, the canonical lifts  $\tilde{\delta}(S)$ of the Dehn twists $\delta(S)$ described above  generate the image of the normalized pure braid group in the mapping class group $\mathrm{Mod}\left(\mathcal{C}_{b},\infty(b)\right)$ when $b  = (2,\ldots,2g)$.

 \begin{corollary} 
 \label{cor:duh}
The map from the normalized pure braid group $\pi_1^{\mathrm{top}}(\mathcal{H}_g^{(2)},b)$ to the mapping class group $\mathrm{Mod}\left(\mathcal{C}_{b},\infty(b)\right)$ can be described as follows.  Pick a loop $\alpha:[0,1] \to \mathcal{H}_g^{(2)}$ based at $b$. Pick a lift $\phi:[0,1] \to \mathcal{G}$ of $\alpha$ such that $\phi(0)$ is the identity.  Continuously lift $\phi(r)$ for $0 \le r \le 1$ to a homeomorphism $\tilde{\phi}(r):\mathcal{C}_b \to \mathcal{C}_{\alpha(r)}$ such that $\tilde{\phi}(0)$ is the identity map.  Then the action of the element $\tilde{z}([\alpha]) \in \mathrm{Mod}\left(\mathcal{C}_{b},\infty(b)\right)$ on $\pi_1^{\mathrm{top}}(\mathcal{C}_b,\infty(b))$  results from letting the homeomorphism  $\tilde{\phi}(1)$ act on loops in $\mathcal{C}_b$ that are based at $\infty(b)$.
 \end{corollary}

 We now show that the action of the normalized pure braid group $\pi_1^{\mathrm{top}}(\mathcal{H}_g^{(2)},b)$ on  $\pi_1^{\mathrm{top}}(\mathcal{C}_b,\infty(b))$ 
 described in Corollary \ref{cor:duh} can also be realized by a conjugation action inside the group 
 $\pi_1^{\mathrm{top}}(\mathcal{C},\infty(b))$.
  
 The projection $\mathcal{C} \to \mathcal{H}_g^{(2)}$ has a canonical section $\infty$ sending
 $t \in \mathcal{H}_g^{(2)}$ to the unique point $\infty(t)$ of the fiber $\mathcal{C}_t$ that lies over the point at infinity on $\mathbb{P}^1(\mathbb{C})$. This has two consequences:
 \begin{itemize}
 \item[(i)] The map $\mathcal{C} \to \mathcal{H}_g^{(2)}$ is a fibration.
 \item[(ii)]  The section $\infty$  induces a homomorphism $\infty_* : \pi_1^{\mathrm{top}}(\mathcal{H}_g^{(2)},b) \to \pi_1^{\mathrm{top}}(\mathcal{C},\infty(b))$ that splits the exact sequence
 \begin{equation}
 \label{eq:conj}
 1 \to \pi_1^{\mathrm{top}}(\mathcal{C}_b,\infty(b)) \to \pi_1^{\mathrm{top}}(\mathcal{C},\infty(b)) \to\pi_1^{\mathrm{top}}(\mathcal{H}_g^{(2)},b) \to 1.
 \end{equation}
 \end{itemize}
 Note that (\ref{eq:conj}) shows that there is an action of $\pi_1^{\mathrm{top}}(\mathcal{H}_g^{(2)},b)$
 on $\pi_1^{\mathrm{top}}(\mathcal{C}_b,\infty(b))$ via conjugation in $\pi_1^{\mathrm{top}}(\mathcal{C},\infty(b))$.  We now show this conjugation action coincides with the action of $\pi_1^{\mathrm{top}}(\mathcal{H}_g^{(2)},b)$ on $\pi_1^{\mathrm{top}}(\mathcal{C}_b,\infty(b))$ via the natural homomorphism $\tilde{z}: \pi_1^{\mathrm{top}}(\mathcal{H}_g^{(2)},b) \to \mathrm{Mod}\left(\mathcal{C}_b,\infty(b)\right)$.
 
\begin{lemma}
\label{lem:oyvey} 
Suppose $\alpha:[0,1] \to \mathcal{H}_g^{(2)}$ is a loop in $\mathcal{H}_g^{(2)}$ based at $b$. Let $\phi$ and $\tilde{\phi}$ be as in Corollary $\ref{cor:duh}$. The pure braid $[\alpha] \in \pi_1^{\mathrm{top}}(\mathcal{H}_g^{(2)},b)$
associated to $\alpha$ acts on the fiber $\mathcal{C}_{b}$ via the canonical lift $\tilde{\phi}(1)$ of $\phi(1) \in \mathcal{G}$.    Suppose $\gamma$ is a loop in the fiber $\mathcal{C}_b$ based at $\infty(b)$. Let $[\gamma]$ be
the class of $\gamma$ in $\pi_1^{\mathrm{top}}(\mathcal{C}_b,\infty(b))$.  Then the image $\tilde{\phi}(1)(\gamma)$
of $\gamma$ under the homeomorphism $\tilde{\phi}(1):\mathcal{C}_b \to \mathcal{C}_b$ is a loop
in $\mathcal{C}_b$ based at $\infty(b)$ that represents the class $\infty_*([\alpha]) \cdot [\gamma] \cdot \infty_*([\alpha])^{-1}$ in $\pi_1^{\mathrm{top}}(\mathcal{C}_b,\infty(b))$ which is the image of $[\gamma]$ under the
element $\tilde{z}([\alpha])\in\mathrm{Mod}\left(\mathcal{C}_b,\infty(b)\right)$ associated to $[\alpha]$.  
In other words, $\tilde{z}([\alpha])$ acts on $\pi_1^{\mathrm{top}}(\mathcal{C}_b,\infty(b))$ as an actual automorphism $($and not just an outer automorphism$)$, and this automorphism of $\pi_1^{\mathrm{top}}(\mathcal{C}_b,\infty(b))$ is given by conjugation by $\infty_*([\alpha])$.
 \end{lemma}
 
 \begin{proof}  Let $\alpha^*: [0,1] \to \mathcal{H}_g^{(2)}$ be the reverse of the loop $\alpha$ defined by $\alpha^*(t) = \alpha(1-t)$ for $ t \in [0,1]$.  It will suffice to show that the loop $\infty_*(\alpha) \circ \gamma \circ\infty_*(\alpha^*)$  based at $\infty(b)$ is homotopic to $\tilde{\phi}(1)(\gamma) \subset \mathcal{C}_b$ inside $\mathcal{C}$. There is a trivial homotopy in $\mathcal{H}_g^{(2)}$ between $\alpha \circ \alpha^*$ and the constant loop based at $b$ described in the following way.  As $t_0$ varies from $1$ down to $0$,
 take the path whose first half is the path described by $\alpha^*(t)$ with $t$ ranging from $0$ to $t_0$, and whose second half is the path described by $\alpha(t)$ with $t$ ranging from $1-t_0$ up to $1$.   When $t_0 = 1$, this is the composite loop $\alpha \circ \alpha^*$, and when $t_0 = 0$ this is the constant loop at $b$.  Now the homeomorphism $\tilde{\phi}(t):\mathcal{C}_{b} \to \mathcal{C}_{\alpha(t)}$  varies continuously with $t$.  As $t_0$ decreases from $1$ to $0$, we  drag the loop $\gamma \subset \mathcal{C}_{\alpha(0)} = \mathcal{C}_b$ back to a loop in $\mathcal{C}_{\alpha(t)}$ based at $\infty(\alpha(t))$ using the homeomorphism 
 $\tilde{\phi}(t)  : \mathcal{C}_{\alpha(0)} = \mathcal{C}_b \to \mathcal{C}_{\alpha(t)}$.  
 These homeomorphisms show that over the trivial homotopy above between $\alpha \circ \alpha^*$ and the constant loop in $\mathcal{H}_g^{(2)}$, there is a homotopy in $\mathcal{C}$ based at $\infty(b)$ between the path $\infty_*(\alpha) \circ \gamma \circ \infty_*(\alpha^*)$ and the loop 
 $\tilde{\phi}(1)(\gamma) \subset \mathcal{C}_b$. This completes the proof.
 \end{proof}

\subsection{Johnson's homomorphisms}
\label{ss:johnson}
We now assume $g > 1$.  
We will apply Morita's results in \cite{Morita1989} on Johnson's homomorphisms to  the split short exact sequence 
$$\xymatrix {
1 \ar[r] & \pi_1^{\mathrm{top}}(\mathcal{C}_b,\infty(b)) \ar[r]\ar@{=}[d] & \pi_1^{\mathrm{top}}(\mathcal{C},\infty(b))\ar[r] &\pi_1^{\mathrm{top}}(\mathcal{H}_g^{(2)},b) \ar@{=}[d]\ar[r] & 1.\\
&\Omega&&G}
$$
from (\ref{eq:conj}). As in Definition \ref{def:lowercentral}(c), let $G(3,\ell)$  be the subgroup of $G=\pi_1^{\mathrm{top}}(\mathcal{H}_g^{(2)},b)$ consisting of elements that act trivially on $\Omega/\Omega_{3,\ell}=\pi_1^{\mathrm{top}}(\mathcal{C}_b,\infty(b))/\pi_1^{\mathrm{top}}(\mathcal{C}_b,\infty(b))_{3,\ell}$. Moreover, let
\begin{equation}
\label{eq:tauforhyper}
\tau_{3,\ell}:G(3,\ell)\to \mathrm{Hom}(\Omega,\Omega_{3,\ell}/\Omega_{4,\ell})
\end{equation}
be the homomorphism from Remark \ref{rem:Johnsonhomgeneral}.

We write
\begin{equation}
\label{eq:pi1}
\Omega=\pi_1^{\mathrm{top}}(\mathcal{C}_b,\infty(b))=\langle x_1,y_1,x_2,y_2,\ldots,x_g,y_g\,:\; [x_1,y_1]\,[x_2,y_2]\,\cdots\,[x_g,y_g]=1\rangle.
\end{equation}

\begin{lemma}
\label{lem:lame1}
Using $(\ref{eq:pi1})$, let $\lambda: \Omega \to \Omega_{3,\ell}/\Omega_{4,\ell}$ be the homomorphism satisfying
\begin{equation}
\label{eq:moritawant!}
\left\{\begin{array}{rccl}
\lambda(x_i) &\equiv& 1\mod \Omega_{4,\ell}, & 1\le i\le g,\\
\lambda(y_1) &\equiv& [[x_1,x_2],x_2]^{-8} \mod\Omega_{4,\ell},\\
\lambda(y_2) &\equiv& [[x_1,x_2],x_1]^8 \mod\Omega_{4,\ell},\\
\lambda(y_i) &\equiv&1\mod \Omega_{4,\ell}, & 3\le i\le g.
\end{array}\right.
\end{equation}
Then there exists an element $\Phi\in G(3,\ell)$ with $\tau_{3,\ell}(\Phi)=\lambda$.
\end{lemma}

\begin{proof}
Let $D_{2g+1}$ be a closed disc $D^2$ with $2g+1$ marked points and boundary $\partial D^2$, where we identify the interior of $D^2$ with $\mathbb{C}$. The mapping class group $\mathrm{Mod}(D_{2g+1})$ is isomorphic to the braid group on $2g+1$ strands (see \cite[\S9.1]{FarbMargalit2012}).

In fact, given a simple closed curve $S$ in $\mathbb{P}^1(\mathbb{C})-\{0,1,2,\ldots,2g,\infty\}=\mathbb{C}-\{0,1,2,\ldots,2g\}$, we can view $S$ as a simple closed curve $S_D$ in $D_{2g+1}$ that stays away from the boundary $\partial D^2$. Taking the homeomorphism of $D^2$ that is the Dehn twist $T_{S_D}$ along $S_D$ and that fixes $\partial D^2$ pointwise, $T_{S_D}$ corresponds to the Dehn twist $\delta(S)\in\mathcal{G}$, and hence to the normalized pure braid described in Lemma \ref{lem:nicepic}(b).

Let $\mathcal{C}_b^0$ be the Riemann surface obtained from $\mathcal{C}_b$ by removing  the interior of a small disc that has $\infty(b)$ at its boundary and that does not enclose any ramification points of $\mathcal{C}_b$. Let $\tilde{T}_{S_D}$ be the canonical lift of the Dehn twist $T_{S_D}$ to $\mathcal{C}_b^0$ that fixes its boundary pointwise, as described in \cite[\S 9.4]{FarbMargalit2012}. Note that by shrinking $S$ homotopically, if necessary, we can assume, without loss of generality, that the canonical lift $\tilde{\delta}(S)$ of the Dehn twist $\delta(S)$ to $\mathcal{C}_b$, as described in Lemma \ref{lem:nicepic}(b), fixes all points, including the boundary, of the small disc whose interior we removed from $\mathcal{C}_b$ to obtain $\mathcal{C}_b^0$. In particular, we obtain that $\tilde{\delta}(S)$ defines a homeomorphism of $\mathcal{C}_b^0$ that fixes its boundary pointwise. This implies that $\tilde{\delta}(S)$ is homotopic to $\tilde{T}_{S_D}$ with respect to a homotopy in $\mathcal{C}_b^0$ that fixes its boundary pointwise. In particular, $\tilde{\delta}(S)$ acts on $\pi_1^{\mathrm{top}}(\mathcal{C}_b^0,\infty(b))$ in the same way as $\tilde{T}_{S_D}$ acts. 

Using the above discussion, it follows from \cite[p. 249]{FarbMargalit2012}, that the lifts  $\tilde{\delta}(S)$ of the Dehn twists $\delta(S)$ described above  generate the image of $G$ in the mapping class group $\mathrm{Mod}\left(\mathcal{C}_{b}^0\right)$. 

Define $\Gamma:=\pi_1^{\mathrm{top}}(\mathcal{C}_b^0,\infty(b))$, which is a free group on $2g$ generators, and write
$$\Gamma= \langle x_1,y_1,x_2,y_2,\ldots,x_g,y_g\rangle.$$

Let $G(3)_\Gamma$ be the subgroup of elements of $G$ acting trivially on $\Gamma/\Gamma_3$. Then we obtain a homomorphism
\begin{equation}
\label{eq:Moritarestriction}
\tau_3: G(3)_\Gamma\to \mathrm{Hom}(\Gamma/\Gamma_2,\Gamma_3/\Gamma_4),
\end{equation}
which we identify with the restriction to $G(3)_\Gamma$ of Johnson's homomorphism for $k=3$ as defined by Morita in \cite[\S1]{Morita1989}. 
Defining $H:=\Gamma/\Gamma_2$, we identify
$$\mathrm{Hom}(\Gamma/\Gamma_2,\Gamma_3/\Gamma_4) = \mathrm{Hom}(H,\Gamma_3/\Gamma_4) = (\Gamma_3/\Gamma_4)\otimes H^* = (\Gamma_3/\Gamma_4)\otimes H.$$
As discussed in \cite[\S1]{Morita1989}, we have a natural isomorphism
$$\Gamma_3/\Gamma_4 = \frac{\wedge^2H \otimes H}{\wedge^3H}$$
which means that we can identify
\begin{equation}
\label{eq:moritatau}
\mathrm{Hom}(\Gamma/\Gamma_2,\Gamma_3/\Gamma_4) =
\frac{\wedge^2H \otimes H}{\wedge^3H}\otimes H.
\end{equation}
Consider the projection
\begin{equation}
\label{eq:moritaproj}
\wedge^2H\otimes\wedge^2 H \subset \wedge^2H\otimes H\otimes H \xrightarrow{\;w\;}\displaystyle\frac{\wedge^2H \otimes H}{\wedge^3H}\otimes H,
\end{equation}
where on simple wedges, the inclusion $\wedge^2 H\subset H\otimes H$ 
is given by
\begin{equation}
\label{eq:wedgetotensor}
h_1\wedge h_2 \mapsto h_1\otimes h_2 - h_2\otimes h_1
\end{equation}
for all $h_1,h_2\in H$.

Let now $p$ be a bounding simple closed curve on $\mathrm{Int}(\mathcal{C}_b^0)$, as depicted in Figure \ref{fig:chickenwithholes}, such that the genus of the subsurface of $\mathcal{C}_b^0$ that $p$ bounds is $1$. As described in \cite[\S 9.4.2, p. 256]{FarbMargalit2012}, the Dehn twist $T_p$ along $p$ is the square of the lift $\tilde{T}_q$ to $\mathcal{C}_b^0$ of the Dehn twist $T_q$ on $D_{2g+1}$ when $q$ is a simple closed curve on $D_{2g+1}$ surrounding the 3 marked points corresponding to the subsurface of $\mathcal{C}_b^0$ of genus 1 that $p$ bounds, as depicted in Figure \ref{fig:chickenwithholes}. 
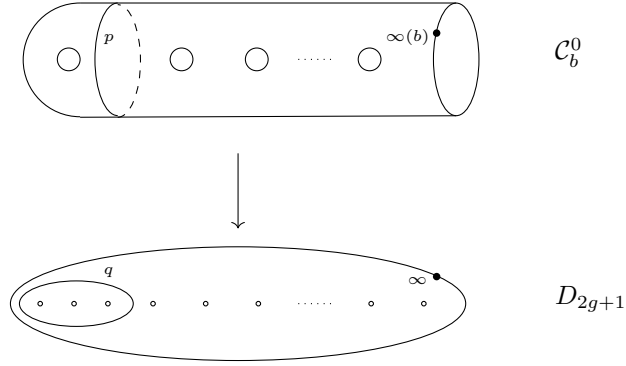
\begin{figure}[h]
\caption{The simple closed curve $p$ on $\mathcal{C}_b^0$ (resp. $q$ on $D_{2g+1}$).}
\label{fig:chickenwithholes}
\[\begin{tikzcd}
  \draw (0,0) arc [start angle=90, end angle=270, radius=5ex];
  \draw (0,0) -- (5,0);
  \draw (0,-1.52) -- (5,-1.5);
  \draw (5,-.75) ellipse (2ex and 4.95ex);
  \node at (4.74,-.4) [circle,fill,inner sep=1pt] {};
  \draw (3.75,-.5) node [label={0:{\mbox{\tiny $\infty(b)$}}}]{};  
  \draw (0,-.75) arc [start angle=0, end angle=360, radius=1ex];
  \draw (1.5,-.75) arc [start angle=0, end angle=360, radius=1ex];
  \draw (2.5,-.75) arc [start angle=0, end angle=360, radius=1ex];
  \draw (4,-.75) arc [start angle=0, end angle=360, radius=1ex];
  \draw [dotted] (2.9,-.75) -- (3.35,-.75);
  \draw (0.5,0) arc(90:270:2ex and 4.95ex);
  \draw (0,-.5) node [label={0:{\mbox{\tiny $p$}}}]{};  
  \draw [dashed] (0.5,-1.52) arc(-90:90:2ex and 4.95ex);
  \draw (6,-.75) node [label={0:\mathcal{C}_b^0}]{};  
  \draw [->] (2.1,-2) -- (2.1,-3);
  \draw (2.1,-4) ellipse (20ex and 5ex);
  \node at (4.74,-3.64) [circle,fill,inner sep=1pt] {};
  \draw (4.04,-3.74) node [label={0:{\mbox{\tiny $\infty$}}}]{};  
  \draw (-.5,-4) arc [start angle=0, end angle=360, radius=.2ex];
  \draw (-.05,-4) arc [start angle=0, end angle=360, radius=.2ex];
  \draw (.4,-4) arc [start angle=0, end angle=360, radius=.2ex];
  \draw (-.05,-4) ellipse (5ex and 2ex);
  \draw (0,-3.6) node [label={0:{\mbox{\tiny $q$}}}]{};  
  \draw (1,-4) arc [start angle=0, end angle=360, radius=.2ex];
  \draw (1.7,-4) arc [start angle=0, end angle=360, radius=.2ex];
  \draw (2.4,-4) arc [start angle=0, end angle=360, radius=.2ex];
  \draw (3.9,-4) arc [start angle=0, end angle=360, radius=.2ex];
  \draw (4.6,-4) arc [start angle=0, end angle=360, radius=.2ex];
  \draw [dotted] (2.9,-4) -- (3.35,-4);
  \draw (6,-4) node [label={0:D_{2g+1}}]{};  
\end{tikzcd}\]
\end{figure}

Let $S_q$ be a corresponding simple closed curve in $\mathbb{P}^1(\mathbb{C})-\{0,1,2,\ldots,2g,\infty\}$ surrounding the 3 branch points corresponding to these 3 marked points, such that the canonical lift $\tilde{\delta}(S_q)$ of the Dehn twist $\delta(S_q)$ defines a homeomorphism of $\mathcal{C}_b^0$ that fixes its boundary pointwise. 

By \cite[\S1]{Morita1989}, $T_p$ acts trivially on $\Gamma/\Gamma_3$, which means that the normalized pure braid associated with $\tilde{\delta}(S_q)$  lies in $G(3)_\Gamma$.
It follows from the proof of \cite[Prop. 1.1]{Morita1989} that $\tau_3(T_p)=\tau_3(\tilde{\delta}(S_q))$ equals the image of $( x_1\wedge y_1)^{\otimes 2}$ under the projection $w$ in (\ref{eq:moritaproj}) and (\ref{eq:wedgetotensor}). In other words,
\begin{equation}
\label{eq:Tp}
w\left(( x_1\wedge y_1)^{\otimes 2}\right)\in\tau_3(G(3)_\Gamma).
\end{equation}

By \cite[p. 3.131]{Mumford}, we have a homomorphism of discrete groups
\begin{equation}
\label{eq:mumford9}
\vartheta:\quad G=\pi_1^{\mathrm{top}}(\mathcal{H}_g^{(2)},b) \to \mathrm{Sp}(2g,\mathbb{Z})/(\pm I)
\end{equation}
given as follows. Let $\alpha:[0,1]\to \mathcal{H}_g^{(2)}$ be a loop based at $b = \alpha(0) = \alpha(1) = (2, 3, \ldots , 2g)$. As before, we lift $\alpha$ to $\phi:[0,1]\to \mathcal{G}$ such that $\pi(\phi(1)) = b$, i.e. $\phi(1)\in K_g$. Hence $\phi(1)$ is a homeomorphism of $\mathbb{P}^1$ carrying $(0,1,\ldots,2g,\infty)$ to itself. We lift $\phi(1)$ to a homeomorphism $\psi$ of the hyperelliptic curve $\mathcal{C}_b$. Then $\psi$ acts on the standard homology basis of $\mathcal{C}_b$ by a $2g\times 2g$ integral symplectic matrix, which we define to be $\vartheta([\alpha])$. By \cite[Lemma 8.12]{Mumford}, we have
\begin{equation}
\label{eq:mumford10}
\mathrm{image}(\vartheta)=\left\{\gamma\mod (\pm I)\;:\; \gamma \in \mathrm{Sp}(2g,\mathbb{Z})\mbox{ with }\gamma\equiv I_{2g}\mod 2\right\}.
\end{equation}
Let now $u=x_1$ and $v=y_1 \pm 2x_2$. Then $u\cdot v = 1$ with respect to the symplectic pairing on $H$ and thus there exists a symplectic automorphism $A$ of $H$ such that $A(x_1) = u$ and $A(y_1) = v$. Moreover, by (\ref{eq:mumford9}) and (\ref{eq:mumford10}), $A$ lies in the image of $G$ under $\vartheta$.
Since $G(3)_\Gamma$ is a normal subgroup of $G$, (\ref{eq:Tp}) implies that
\begin{equation}
\label{eq:ATp}
w\left((x_1\wedge (y_1 \pm 2x_2))^{\otimes 2}\right)=
w\left(( u\wedge v)^{\otimes 2}\right) \in \tau_3(G(3)_\Gamma).
\end{equation}
As in the proof of \cite[Prop. 1.2]{Morita1989}, we see that
\begin{eqnarray*}
( x_1\wedge (y_1 + 2x_2))^{\otimes 2} &=& ( x_1\wedge y_1)^{\otimes 2} + 4\cdot ( x_1\wedge x_2)^{\otimes 2} + 2 \cdot (x_1\wedge y_1 \leftrightarrow x_1\wedge x_2),\\
( x_1\wedge (y_1 - 2x_2))^{\otimes 2} &=& ( x_1\wedge y_1)^{\otimes 2} + 4\cdot ( x_1\wedge x_2)^{\otimes 2} - 2 \cdot (x_1\wedge y_1 \leftrightarrow x_1\wedge x_2),
\end{eqnarray*}
where
$$(x_1\wedge y_1 \leftrightarrow x_1\wedge x_2 ) = (x_1\wedge y_1) \otimes (x_1\wedge x_2) +  (x_1\wedge x_2) \otimes (x_1\wedge y_1).$$
By (\ref{eq:Tp}) and (\ref{eq:ATp}), this implies that 
$$w\left(8\cdot ( x_1\wedge x_2)^{\otimes 2}\right)\in\tau_3(G(3)_\Gamma).$$
Using (\ref{eq:moritaproj}) and (\ref{eq:wedgetotensor}), we see that
$$w\left(8\cdot(x_1\wedge x_2)^{\otimes 2}\right) = 8\cdot (x_1\wedge x_2) \otimes \left((x_1+\wedge^3H)\otimes x_2 - (x_2+\wedge^3H)\otimes x_1\right).$$
Using the identification (\ref{eq:moritatau}), this means that $w\left(8\cdot(x_1\wedge x_2)^{\otimes 2}\right)$ corresponds to the homomorphism 
$$\lambda_1: H=\Gamma/\Gamma_2 \to \Gamma_3/\Gamma_4$$ 
satisfying
\begin{equation}
\label{eq:moritaalmost}
\left\{\begin{array}{rccl}
\lambda_1(x_i) &\equiv& 1\mod \Gamma_4, & 1\le i\le g,\\
\lambda_1(y_1) &\equiv& [[x_1,x_2],x_2]^{-8} \mod\Gamma_4,\\
\lambda_1(y_2) &\equiv& [[x_1,x_2],x_1]^8 \mod\Gamma_4,\\
\lambda_1(y_i) &\equiv&1\mod \Gamma_4, & 3\le i\le g.
\end{array}\right.
\end{equation}

We can identify $\mathrm{Mod}\left(\mathcal{C}_{b}^0\right)$ with all automorphisms of $\Gamma$ that fix the element
$$c:=[x_1,y_1]\,[x_2,y_2]\,\cdots\,[x_g,y_g]\in \Gamma_2.$$
Therefore, all elements of $\mathrm{Mod}\left(\mathcal{C}_{b}^0\right)$ induce actual automorphisms of $\Omega=\Gamma/\langle c\rangle$. Restricting this action to the image of $G$ in $\mathrm{Mod}\left(\mathcal{C}_{b}^0\right)$, as described at the beginning of the proof, we obtain that this restricted action coincides with the action of $G=\pi_1^{\mathrm{top}}(\mathcal{H}_g^{(2)},b)$ on $\Omega=\pi_1^{\mathrm{top}}(\mathcal{C}_b,\infty(b))$ as in Lemma \ref{lem:oyvey}.

Since $\Omega=\Gamma/\langle c\rangle$ and $\Omega_3\le \Omega_{3,\ell}$, we have natural projections
$$\Gamma/\Gamma_3\to \Omega/\Omega_3\quad\mbox{and}\quad\Omega/\Omega_3\to \Omega/\Omega_{3,\ell}.$$
This implies
\begin{equation}
\label{eq:ohno!}
G(3)_\Gamma \le G(3)_\Omega\le G(3,\ell),
\end{equation}
when $G(3)_\Omega$ is the subgroup of $G$ consisting of elements acting trivially on $\Omega/\Omega_3$. 
Since $\lambda_1$ from (\ref{eq:moritaalmost}) lies in $\tau_3(G(3)_\Gamma)$, there exists $\Phi\in G(3)_\Gamma$ such that $\tau_3(\Phi)=\lambda_1$. Since $\Phi$ also lies in $G(3,\ell)$ by (\ref{eq:ohno!}), we conclude from (\ref{eq:moritaalmost}) that the homomorphism $\lambda:\Omega \to \Omega_{3,\ell}/\Omega_{4,\ell}$ from (\ref{eq:moritawant!}) equals $\tau_{3,\ell}(\Phi)$. 
\end{proof}

\begin{corollary}
\label{cor:nonvanishingtripleMassey}
Using the notation from Lemma $\ref{lem:lame1}$, and in particular $(\ref{eq:pi1})$, define three characters $\bar{\chi}_1,\bar{\chi}_2,\bar{\chi}_3\in\mathrm{Hom}(\Omega,\mathbb{Z}/\ell)$ by
$$\bar{\chi}_1=\bar{\chi}_3=x_1^*\quad\mbox{and} \quad\bar{\chi}_2=x_2^*.$$
Then $\bar{\chi}_1,\bar{\chi}_2,\bar{\chi}_3$ satisfy conditions $(i)$, $(ii)$ and $(iii)$ of Proposition $\ref{prop:sufficient_general}$ for $\omega_0=y_2$ and $\Phi\in G(3,\ell)$ as in Lemma $\ref{lem:lame1}$.
\end{corollary}

\begin{proof}
Define $\rho_\Omega:\Omega \to U_4(\mathbb{Z}/\ell)$ by
\begin{eqnarray*}
\rho_\Omega(x_1)&=&\left(\begin{array}{cccc}1&1&0&0\\0&1&0&0\\0&0&1&1\\0&0&0&1\end{array}\right),\\
\rho_\Omega(x_2) &=& \left(\begin{array}{cccc}1&0&0&0\\0&1&1&0\\0&0&1&0\\0&0&0&1\end{array}\right),\quad\mbox{and}\\
\rho_\Omega(z)&=&\left(\begin{array}{cccc}1&0&0&0\\0&1&0&0\\0&0&1&0\\0&0&0&1\end{array}\right)\quad
\mbox{for all $z\in\{x_3,\ldots,x_g,y_1,\ldots,y_g\}$.}
\end{eqnarray*}
Since commutators of group elements with the identity element are equal to the identity element, it follows that $\rho_\Omega$ is a group homomorphism. Because $\rho_\Omega$ is associated with $(\bar{\chi}_1,\bar{\chi}_2,\bar{\chi}_3)$ in the sense that the entries above the diagonal are given by $\bar{\chi}_1,\bar{\chi}_2,\bar{\chi}_3$, it follows by Remark \ref{rem:tripleMassey}(b) that the triple Massey product $\langle \bar{\chi}_1,\bar{\chi}_2,\bar{\chi}_3\rangle$ contains 0, which is condition (i) of Proposition \ref{prop:sufficient_general}.

None of the $\bar{\chi}_i$ are zero and both sets $\{\bar{\chi}_1,\bar{\chi}_2\}$ and $\{\bar{\chi}_2,\bar{\chi}_3\}$ are $(\mathbb{Z}/\ell)$-linearly independent. Hence we also have condition (ii) of Proposition \ref{prop:sufficient_general}.

For condition (iii), we first note that the definition of $\bar{\chi}_1,\bar{\chi}_2,\bar{\chi}_3$ implies that
$$\bar{\chi}_1(y_2)=\bar{\chi}_2(y_2)=\bar{\chi}_3(y_2)=0.$$
Suppose next that $\sigma,\tau,\gamma \in \{x_1,y_1,\ldots,x_g,y_g\}$ and $c_{\sigma,\tau,\gamma} \ne 0$ in (\ref{eq:triplecomm}).  Then $\gamma = x_1$, $\sigma , \tau \in \{x_1,x_2\}$ and $\sigma \ne \tau$ since otherwise the matrix on the right side of (\ref{eq:triplecomm}) either has a row that is zero or two equal rows. Therefore, the homomorphism $h_\ell:\Omega_{3,\ell}\to \mathbb{Z}/\ell$ from Remark \ref{rem:wretched}(b) satisfies
$$h_\ell([[\sigma,\tau],\gamma]) = \left\{\begin{array}{rl}
2 & \mbox{if $(\sigma,\tau,\gamma) = (x_1,x_2,x_1)$},\\
-2 & \mbox{if $(\sigma,\tau,\gamma) = (x_2,x_1,x_1)$},\\
0&\mbox{if $\sigma,\tau,\gamma \in \{x_1,y_1,\ldots,x_g,y_g\}$ and $(\sigma,\tau,\gamma) \not\in \{(x_1,x_2,x_1), (x_2,x_1,x_1)\}$}.
\end{array}\right.$$
Letting $\Phi$ and $\lambda$ be as in Lemma \ref{lem:lame1}, we obtain
$$h_\ell(\tau_{3,\ell}(\Phi)(y_2))=h_\ell(\lambda(y_2))=
h_\ell([[x_1,x_2],x_1]^8) = 8\cdot h_\ell([[x_1,x_2],x_1])=16\ne 0$$
since $\ell >3$. Therefore, we also have condition (iii) of Proposition \ref{prop:sufficient_general}, which concludes the proof of Corollary \ref{cor:nonvanishingtripleMassey}.
\end{proof}


\section{Algebraization and specialization}
\label{s:algebraization}

In this section we consider the fibration of hyperelliptic curves $\mathcal{C}\to\mathcal{H}_g^{(2)}$ defined in \S{}\ref{ss:hyperell}, and its section $\infty:\mathcal{H}_g^{(2)}\to \mathcal{C}$ sending a point $t\in \mathcal{H}_g^{(2)}$ to the unique point at infinity of the fiber $\mathcal{C}_t$.

For any fixed $t\in \mathcal{H}_g^{(2)}$, we have an exact sequence of topological fundamental groups attached to the fibration 
\begin{equation}
\label{eq:topologicalpi1}
1\to \pi_1^{\mathrm{top}}(\mathcal{C}_t,\infty(t)) \to \pi_1^{\mathrm{top}}(\mathcal{C},\infty(t)) \to \pi_1^{\mathrm{top}}(\mathcal{H}_g^{(2)},t) \to 1.
\end{equation}
The section $\infty$ gives a splitting $\infty_*$ of this exact sequence. In other words, 
$$
\pi_1^{\mathrm{top}}(\mathcal{C},\infty(t)) = \pi_1^{\mathrm{top}}(\mathcal{C}_t,\infty(t)) \rtimes_\infty \pi_1^{\mathrm{top}}(\mathcal{H}_g^{(2)},t),
$$
where the label $\infty$ emphasizes the fact that the action of any $\Phi\in \pi_1^{\mathrm{top}}(\mathcal{H}_g^{(2)},t)$ on $\pi_1^{\mathrm{top}}(\mathcal{C},\infty(t))$ is given by conjugation with $\infty_*(\Phi)$.
Let us observe that this semidirect product does not depend on the choice of the base point $t$ in the following sense.  If we choose another base point $t'$,  the isomorphisms induced on the  fundamental groups of $\mathcal{H}_g^{(2)}$ and $\mathcal{C}$ by the change of base point $t\mapsto t'$ and $\infty(t)\mapsto \infty(t')$ commute with the formation of the exact sequence \eqref{eq:topologicalpi1}.  
Therefore the resulting semidirect products are canonically isomorphic. In particular, choosing $t'=b$ as in the previous section, one obtains a canonical isomorphism
\begin{equation}
\label{eq:semidirectpi1top}
\pi_1^{\mathrm{top}}(\mathcal{C},\infty(t)) \cong \Omega \rtimes_\infty G
\end{equation}
where $G$ and $\Omega$ are as in \S{}\ref{ss:johnson}.
It is relevant to consider other base points than $b$ in view of the arithmetic constructions we will consider later.

\subsection{Algebraization}
\label{ss:alg}

We will make extensive use of the following algebraization result, which follows by combining \cite[exp. XII, Th{\'e}or{\`e}me 5.1]{SGA1} and \cite[exp. XIII, Proposition 4.6]{SGA1}.

\begin{theorem}
\label{thm:JeanAlgebra}
Let $X$ be a scheme of finite type over $\bar{\Q}$. Then there is an equivalence of categories between finite coverings of the topological space $X(\C)$ and finite {\'e}tale covers of $X$ defined over $\bar{\Q}$.
\end{theorem}

Since any scheme of finite type over $\bar{\Q}$ can be defined over some number field $K$, it follows that given a finite (topological) covering of $X(\C)$ there exists a number field $K$ and an {\'e}tale cover of $X$ defined over $K$ from which the topological covering is induced.

As one can see from their definition, the topological spaces $\mathcal{H}_g^{(2)}$ and $\mathcal{C}$ are the complex points of algebraic varieties $\mathcal{H}_{g,\Q}^{(2)}$ and $\mathcal{C}_{\Q}$  defined over $\Q$. More precisely, $\mathcal{H}_{g,\Q}^{(2)}$ is a Zariski open subset of the affine plane $\mathbb{A}^{2g-1}$ over $\Q$ (in particular it is a rational variety), and the map $\mathcal{C}_{\Q} \to \mathcal{H}_{g,\Q}^{(2)}$ is a fibration of smooth projective hyperelliptic curves of genus $g$.

We denote by $B(i,\ell) \to \mathcal{H}_g^{(2)}$ the finite topological coverings of $\mathcal{H}_g^{(2)}$ corresponding (under the Galois correspondence for covering spaces) to the finite index subgroups $G(i,\ell)$ of $G$ (see Definition \ref{def:lowercentral}(c) and Remark \ref{rem:lowercentral}(b)). 

On applying Theorem \ref{thm:JeanAlgebra}, we arrive at a finite \'etale cover of varieties 
$$
B(4,\ell)_K\to \mathcal{H}_{g,K}^{(2)}
$$
over a number field $K$ whose complex points are identified with the cover $B(4,\ell) \to \mathcal{H}_g^{(2)}$. One may require $K$ to be as small as possible but this is not needed in order to prove our specialization result. Since $B(3,\ell)$ and $B(2,\ell)$ are subcovers of $B(4,\ell)$, these are also defined over $K$.

\begin{remark}
\label{rem:legendrecurve}
Although $\mathcal{H}_g^{(2)}$ is defined over $\Q$, the case when $\mathcal{C}$ is a fibration of elliptic curves ($g=1$) already shows that the $B(i,\ell)$ do not descend to $\Q$; more precisely their field of definition $K$ contains $\Q(\zeta_\ell)$. In this case the cover $\mathcal{C} \to \mathcal{H}_1^{(2)} = \mathbb{P}^1-\{0,1,\infty\}$ has affine model the Legendre family $y^2 = x \cdot (x - 1) \cdot (x - T)$.  The cover $B(2,\ell)\to \mathcal{H}_1^{(2)}$ is a Galois cover with group $\SL_2(\Z/\ell)$, as was shown by Igusa \cite{Igusa1959}.  Over $\Q(\mathbb{P}^1)=\Q(T)$  the Galois representation is surjective onto $\GL_2(\Z/\ell)$, so the field of definition of the cover $B(2,\ell)$ is exactly $\Q(\zeta_\ell)$, via the cyclotomic character. 
\end{remark}

Suppose now that $X$ is any algebraic variety defined over $K$ and that we fix an embedding $\iota:K \hookrightarrow \mathbb{C}$.  Given a point $x \in X(K)$, we apply $\iota$ to arrive at a 
geometric point $\bar{x}:\mathrm{Spec}(\mathbb{C}) \to X$.  We will also use $\bar{x}$ to denote the point of the topological space $X(\mathbb{C})$ that results from $x$ and $\iota$.   We have a comparison morphism $\pi_1^{\mathrm{top}}(X(\mathbb{C}),\bar{x})\to \pi_1(X,\bar{x})$ between topological and {\'e}tale fundamental groups.  This comparison morphism factors through the profinite completion 
$\widehat{\pi_1^{\mathrm{top}}(X(\mathbb{C}),\bar{x})}$ and is functorial with respect to $X$, $x$ and 
$\iota$.

Now let $t$ be a $K$-rational point of $\mathcal{H}_{g,K}^{(2)}$.  This gives a geometric point 
$\bar{t}:\mathrm{Spec}(\mathbb{C}) \to  \mathcal{H}_{g,K}^{(2)}$. 
Since $\mathcal{C}_K\to \mathcal{H}_{g,K}^{(2)}$ has a section and the base field $K$ has characteristic zero, it follows from \cite[exp. XIII, Exemples 4.4]{SGA1} that we have an exact sequence of {\'e}tale fundamental groups
\begin{equation}
\begin{CD}
1 @>>> \pi_1(\mathcal{C}_{\bar{t}},\infty(\bar{t})) @>>> \pi_1(\mathcal{C}_K,\infty(\bar{t})) @>>> \pi_1(\mathcal{H}_{g,K}^{(2)},\bar{t}) @>>> 1 \\
\end{CD}
\end{equation}
and a splitting $\infty_*$ of this exact sequence. We write the corresponding semidirect product as
$$
\pi_1(\mathcal{C}_K,\infty(\bar{t})) = \pi_1(\mathcal{C}_{\bar{t}},\infty(\bar{t})) \rtimes_\infty \pi_1(\mathcal{H}_{g,K}^{(2)},\bar{t}).
$$
By the discussion above, we have a morphism
$$
\widehat{\pi_1^{\mathrm{top}}(\mathcal{C},\infty(\bar{t}))} \to \pi_1(\mathcal{C}_K,\infty(\bar{t}))
$$
which, by naturality, respects the structure of semidirect product corresponding to the section $\infty$ on both sides. We can decompose this morphism  as follows
\begin{equation}
\label{eq:comparison2}
\widehat{\pi_1^{\mathrm{top}}(\mathcal{C},\infty(\bar{t}))} \cong \widehat{\Omega} \rtimes_\infty \widehat{G} \cong  \pi_1(\mathcal{C}_{\bar{t}},\infty(\bar{t})) \rtimes_\infty \widehat{G} \to \pi_1(\mathcal{C}_{\bar{t}},\infty(\bar{t})) \rtimes_\infty \pi_1(\mathcal{H}_{g,K}^{(2)},\bar{t})
\end{equation}
where the first isomorphism comes from Remark \ref{rem:profcompl2}, the second isomorphism is by \cite[exp. XII, Corollaire 5.2]{SGA1}, and the last morphism is induced by the comparison morphism
$$
\widehat{G}=\widehat{\pi_1^{\mathrm{top}}(\mathcal{H}_g^{(2)},\bar{t})} \to \pi_1(\mathcal{H}_{g,K}^{(2)},\bar{t}).
$$

Let $\mathcal{H}_{g,\bar{K}}^{(2)}$ be $\bar{K} \otimes_K \mathcal{H}_{g,K}^{(2)}$.
The cover $B(4,\ell)_K\to \mathcal{H}_{g,K}^{(2)}$ is defined over $K$, and its complex points define the topological covering $B(4,\ell) \to \mathcal{H}_g^{(2)}$ with group $G/G(4,\ell)$.  Therefore the cover 
$B(4,\ell)_K\to \mathcal{H}_{g,K}^{(2)}$ has group $G/G(4,\ell) = \widehat{G}/\widehat{G}(4,\ell)$,
and its base change from $K$ to $\bar{K}$ is irreducible.  We conclude that the Galois covers 
$B(4,\ell)_K\to \mathcal{H}_{g,K}^{(2)}$ and $\mathcal{H}_{g,\bar{K}}^{(2)} \to \mathcal{H}_{g,K}^{(2)}$ are disjoint.  Therefore the subgroup of $\pi_1(\mathcal{H}_{g,K}^{(2)},\bar{t})$  that fixes
$B(4,\ell)_K$ surjects onto its quotient 
$\mathrm{Gal}(\bar{K}/K) = \pi_1(\mathcal{H}_{g,K}^{(2)},\bar{t})/\pi_1(\mathcal{H}_{g,\bar{K}}^{(2)},\bar{t})$.  
Hence, the morphism \eqref{eq:comparison2} induces an isomorphism
\begin{equation}
\label{eq:comparison3}
\overline{\Omega} \rtimes_\infty \overline{G} \cong \overline{\pi_1(\mathcal{C}_{\bar{t}},\infty(\bar{t}))} \rtimes_\infty \overline{\pi_1(\mathcal{H}_{g,K}^{(2)},\bar{t})}
\end{equation}
where the bars on the groups are the notation introduced in Remark \ref{rem:crucial}. Note that we have replaced $\widehat{G}/\widehat{G}(4,\ell)$ with $G/G(4,\ell)$ and similarly for $\Omega/\Omega_{4,\ell}$, which is made possible by Remark \ref{rem:profcompl2}.

\subsection{Specialization}
\label{ss:spe}

As previously, let $t$ be a $K$-rational point of  $\mathcal{H}_{g,K}^{(2)}$ with corresponding geometric point $\bar{t}$. Then we have a commutative diagram with exact rows
\begin{equation}
\label{eq:specialization}
\begin{CD}
1 @>>> \pi_1(\mathcal{C}_{\bar{t}},\infty(\bar{t})) @>>> \pi_1(\mathcal{C}_t,\infty(\bar{t})) @>>> \Gal(\bar{K}/K) @>>> 1 \\
@. @| @VVt_*V @VVt_*V \\
1 @>>> \pi_1(\mathcal{C}_{\bar{t}},\infty(\bar{t})) @>>> \pi_1(\mathcal{C},\infty(\bar{t})) @>>> \pi_1(\mathcal{H}_{g,K}^{(2)},\bar{t}) @>>> 1 \\
\end{CD}
\end{equation}
in which the vertical maps $t_*$ are functorially induced by the map $t:\Spec(K)\to \mathcal{H}_{g,K}^{(2)}$.

We consider these exact sequences as semidirect products, the splittings being induced by the section at infinity of $\mathcal{C}$ as usual.

We apply Hilbert's Irreducibility Theorem \cite[Chap. 3]{SerreTopics} to the cover $B(4,\ell)_K \to \mathcal{H}_{g,K}^{(2)}$. Since $\mathcal{H}_{g,K}^{(2)}$ is a rational variety, we obtain the existence of infinitely many $t\in \mathcal{H}_{g,K}^{(2)}(K)$ such that the fiber of $B(4,\ell)_K$ above $t$ remains irreducible.
This fiber is therefore the spectrum of a number field $K(4,\ell)(t)$ whose Galois group over $K$ is the same as the Galois group of $B(4,\ell)_K \to \mathcal{H}_{g,K}^{(2)}$, which is none other than $\overline{G}$ as we have seen in \S{}\ref{ss:alg}. By the same reasoning as in \S{}\ref{ss:alg} and in combination with \eqref{eq:comparison3}, we conclude that we have, for such $t$, an isomorphism
\begin{equation}
\label{eq:comparison4}
\overline{\Omega} \rtimes_\infty \overline{G} \cong \overline{\pi_1(\mathcal{C}_{\bar{t}},\infty(\bar{t}))} \rtimes_\infty \overline{\Gal(\bar{K}/K)}.
\end{equation}

We now prove our main result, which implies Theorem 
\ref{thm:TedsShamefulVersion} of the introduction.

\begin{theorem}
\label{thm:JeanVersion}
Let $\ell>3$ be a prime number, and let $g>1$ be an integer. Then there exist a number field $K$ containing $\Q(\zeta_\ell)$ and infinitely many $($non-isomorphic$)$ genus $g$ hyperelliptic curves $X$ defined over $K$, with the following property.  For each such $X$, there exists a finite extension $F/K$, over which the $\ell$-torsion $\mathrm{Pic}(\bar{X})[\ell]$ is defined, together with characters $\chi_1,\chi_2,\chi_3\in \HH^1(X\otimes_K F,\mathbb{Z}/\ell)$ whose triple Massey product is not empty and does not contain zero.
\end{theorem}

\begin{proof}
Consider the semidirect product $\Omega\rtimes_\infty G$ from \eqref{eq:semidirectpi1top}. According to Corollary \ref{cor:nonvanishingtripleMassey}, there exist characters of $\Omega$ satisfying conditions  (i), (ii) and (iii) of Proposition \ref{prop:sufficient_general}. By Remark \ref{rem:crucial}, it is equivalent to consider the corresponding data for the quotient groups $\overline{\Omega}$ and $\overline{G}$. As previously, let $K$ be the number field over which $B(4,\ell)$ is defined (this field $K$ contains $\Q(\zeta_\ell)$ by Remark \ref{rem:legendrecurve}). By \eqref{eq:comparison4}, there exist infinitely many values of $t\in \mathcal{H}_{g,K}^{(2)}(K)$ such that the curve $\mathcal{C}_t$ satisfies
$$
\overline{\Omega} \rtimes_\infty \overline{G} \cong \overline{\pi_1(\mathcal{C}_{\bar{t}},\infty(\bar{t}))} \rtimes_\infty \overline{\Gal(\bar{K}/K)}.
$$

Under this isomorphism, the characters of $\Omega$ are mapped to characters of $\pi_1(\mathcal{C}_{\bar{t}},\infty(\bar{t}))$ satisfying conditions (i), (ii) and (iii) of Proposition \ref{prop:sufficient_general}. In order to ensure that these characters of $\pi_1(\mathcal{C}_{\bar{t}},\infty(\bar{t}))$ extend to characters of $\pi_1(\mathcal{C}_t,\infty(\bar{t}))$ whose triple Massey product is not empty, we base change the curve $\mathcal{C}_t$ to the field $F$ fixed by $\Gal(\bar{K}/K)(3,\ell)$ (equivalently, $F$ is the field obtained by specializing the cover $B(3,\ell)_K\to \mathcal{H}_{g,K}^{(2)}$ above $t$). The result then follows from the conclusion of Proposition \ref{prop:sufficient_general}. Note that the $\ell$-torsion $\mathrm{Pic}(\mathcal{C}_{\bar{t}})[\ell]$ is defined over $F$ since $\Gal(\bar{K}/K)(2,\ell)$ contains $\Gal(\bar{K}/K)(3,\ell)$.

The fact that there are infinitely many non-isomorphic curves in our family is because for a given $t\in \mathcal{H}_{g,K}^{(2)}(K)$ there are exactly $(2g-1)!$ values of $t'\in \mathcal{H}_{g,K}^{(2)}(K)$ such that $\mathcal{C}_t\cong \mathcal{C}_{t'}$ as hyperelliptic curves (this corresponds to the action of the symmetric group on $\mathcal{H}_{g,K}^{(2)}$). Hence an infinite family of $t$ gives infinitely many non-isomorphic curves.
\end{proof}

\begin{remark}
\label{rem:Chebotarev}
In order to get examples over finite fields, we apply Chebotarev to the Galois extension $K(4,\ell)(t)/F$. At the same time, one keeps control on the \'etale fundamental group of the curve over the algebraic closure of the finite field provided the characteristic $p$ is prime to $\ell$, since the ``prime-to-$p$'' part of the \'etale fundamental group remains the same as in the characteristic $0$ case.  This means that conditions (i) and (ii) of Proposition \ref{prop:sufficient_general} hold for smooth specializations of $X\otimes_K F$ at finite places of $F$ that are prime to $\ell$ and that are unramified in $K(4,\ell)(t)$.  We know that condition (iii) of Proposition \ref{prop:sufficient_general} for $X\otimes_K F$ can
be satisfied by some pair of elements $(\Phi,\omega_0)$ as in the statement of the Proposition.  We now keep the same $\omega_0$ and use a Chebotarev density theorem to say $\Phi$ can be realized as  the Frobenius automorphism in $\mathrm{Gal}(K(4,\ell)(t)/F)$ associated to a place of $K(4,\ell)(t)$ of residue characteristic prime to $\ell$.
\end{remark} 


\end{document}